\documentclass[11pt,reqno,a4paper]{amsart}

\usepackage{array}
\usepackage{enumitem}
\usepackage{mathtools}
\usepackage{pgfplots}
\usepackage{soul}
\usepackage{bm} 
\usepackage{mathrsfs} 
\usepackage{pifont} 
\usepackage{graphicx} 

\usepackage{ulem} 

\usepackage{amsmath,amssymb,amsfonts}
\usepackage{tikz}
\usepackage{hyperref}
\usepackage{appendix}
\usepackage{scalerel}
\usepackage{graphicx}
\usepackage{subcaption}
\usepackage{amsthm}
\usepackage{multirow}
\usepackage{pdflscape}
\usepackage{afterpage}
\usepackage{capt-of} 
\usepackage{lipsum}
\usepackage{booktabs}
\usepackage{siunitx}
\usepackage{esvect}

\tikzset{decorated arrows/.style={
    postaction={
        decorate,
        decoration={
            markings,
            mark=between positions 0 and 1 step 15mm with {\arrow[black]{stealth};}
            }
        },
    }
}
 
\tikzset{decorated arrows2/.style={
    postaction={
        decorate,
        decoration={
            markings,
            mark=at position 15mm with {\arrow[black]{stealth};}
            }
        },
    }
}

\usepgfplotslibrary{colormaps}
\usetikzlibrary{arrows}

\makeatletter
\def\set@curr@file#1{%
  \begingroup
    \escapechar\m@ne
    \xdef\@curr@file{\expandafter\string\csname #1\endcsname}%
  \endgroup
}
\def\quote@name#1{"\quote@@name#1\@gobble""}
\def\quote@@name#1"{#1\quote@@name}
\def\unquote@name#1{\quote@@name#1\@gobble"}
\makeatother
\usepackage{graphics}

\theoremstyle{plain}
\newtheorem{thm}{Theorem}
\newtheorem{lem}{Lemma}
\newtheorem{prop}{Proposition}
\newtheorem{cor}[thm]{Corollary}
\newtheorem{maintheorem}{Theorem}

 \theoremstyle{definition}
\newtheorem{defn}{Definition}

\theoremstyle{remark}
\newtheorem*{rem}{Remark}

\theoremstyle{plain}

\newcommand{\NN}{{\mathbb{N}}}

\newcommand{\RR}{\mathbb{R}}

\newcommand{\dpt}{\displaystyle}

\newcommand{\RN}[1]{%
  \textup{\uppercase\expandafter{\romannumeral#1}}%
}

\author[J. P. S. M. de Carvalho and A. A. Rodrigues]{Jo\~ao P. S. Maur\'icio de Carvalho$^{1}$ and Alexandre A. Rodrigues$^{2}$ \\
\\
$^1$\MakeLowercase{jocarvalho@fc.up.pt} \\ $^2$\MakeLowercase{alexandre.rodrigues@fc.up.pt, arodrigues@iseg.ulisboa.pt} \\
\\
$^{1}$Faculty of Sciences, University of Porto,  \\
$^{1}$Centre for Mathematics, University of Porto, \\Rua do Campo Alegre s/n, Porto 4169-007, Portugal \\
$^{2}$Lisbon School of Economics and Management, \\ Centro de Matem\'atica Aplicada \`a Previs\~{a}o e Decis\~{a}o Econ\'omica\\ Rua do Quelhas 6, 1200-781 Lisboa, Portugal \\
}

\begin{document}



\thanks{The first author was supported by CMUP, Portugal (UIDP/MAT/00144/2020), which is funded by Funda\c{c}\~ao para a Ci\^encia e a Tecnologia (FCT), with national and European structural funds through the programs FEDER, under the partnership agreement PT2020. The second author  was supported by the Project CEMAPRE/REM -- UIDB /05069/2020 financed by FCT/MCTES through national funds.\\ $^1$Corresponding author.}

\title[Pulse vaccination in a  SIR model]
{Pulse vaccination in a  SIR model:\\  global dynamics, bifurcations  and seasonality}

\date{\today}  
 
\begin{abstract}
 We analyze a periodically-forced  dynamical system inspired by the SIR model with impulsive vaccination. We fully characterize its dynamics according to the proportion $p$ of  vaccinated individuals and the time $T$ between doses.   If the {\it basic reproduction number} is less than 1 ({\it i.e.} $\mathcal{R}_p<1$), then we obtain precise conditions for the existence and global stability of a disease-free   {\it $T$-periodic} solution. Otherwise, if $\mathcal{R}_p>1$, then a globally stable  {\it $T$-periodic} solution emerges with  positive coordinates.
 
  We draw a bifurcation diagram $(T,p)$ and we describe the associated bifurcations.
 We also find analytically and numerically chaotic dynamics by adding seasonality to the disease transmission rate.  In a realistic context, low vaccination coverage and intense seasonality may result in unpredictable dynamics.
 Previous experiments have suggested chaos in periodically-forced biological impulsive models, but no analytic proof has been given.
 \end{abstract}

\maketitle


\section{Introduction}
Mathematical models have proved to be a useful tool in epidemiology, providing insights into the dynamics of infectious diseases and giving insights to improve strategies to combat their spread \cite{Hethcote2000, Cobey2020}. In particular, the SIR model, which classifies individuals into {\it Susceptibles} ($S$), {\it Infectious} ($I$) and {\it Recovered} ($R$), has been extensively used \cite{KermackMcKendrick1932,BrauerChavez2012}.

A prompt response is crucial once a disease   emerges in a population. A range of approaches have been explored \cite{Diekmann2012,DriesscheWatmough2002,CarvalhoPinto2021}, with vaccination being one of the most effective ways to stop the  progression \cite{Plotkin2005,Plotkin2011}. 

Among the vaccination strategies in SIR models, we can point out three types:

\begin{enumerate}
\item {\it constant vaccination}, where a fixed proportion of the population is vaccinated \cite{CarvalhoRodrigues2023,Makinde2007,Shulgin1998}, as   in the administration of the BCG vaccine against tuberculosis to newborns; 
\item {\it pulse vaccination}, which involves vaccinating a percentage of the {\it Susceptible} individuals periodically \cite{Shulgin1998,Agur1993,Shulgin2000,LuChiChen2002}, similar to mass vaccination campaigns using oral polio vaccine in areas with polio outbreaks; and 
\item {\it mixed vaccination} \cite{Shulgin1998}, as the hepatitis B vaccination  program, which starts with one shot immediately after birth, followed by subsequent shots.
\end{enumerate}

Identifying the most suitable vaccination strategy represents a challenge, requiring considerations of effectiveness and costs associated with public health policies. 
 Although considerable literature considers the logistic growth of the susceptible population, seasonality and the effect of vaccination, the combination of them remains unexplored.

Pulse vaccination over a proportion $p$ of the population may be a strategy to stop the progression of an infectious disease \cite{Shulgin1998,LuChiChen2002,Meng2008,JinHaque2007}.
The effective strategy   must be in a way that the proportion $p$ of vaccination is at the target level needed for the disease eradication, and the time $T$ between doses must be suitable.
Finding the most appropriate pair $(T,p)$ for specific models is an open problem.

Some authors highlight the presence of seasonal forces in epidemic models such as school holidays, climate change, and political decisions \cite{Buonomo2018}. 
While the seasonal impact is negligible for some diseases, for others such as childhood illnesses and influenza, it may have dramatic consequences in the dynamics of the models {\cite{CarvalhoRodrigues2022}.
Differential equations adjust transmission rates using periodic functions \cite{Keeling2001}, making them more complex than standard models but also more realistic \cite{Barrientos2017,DuarteJanuario2019}.  

This work focuses on the application of a modified SIR model with pulse vaccination and subject to seasonality,  a promising unexplored field.

\subsection*{State of the art}\footnote{ A wide range of epidemiological models address the impact of pulse vaccination. In the section ``state of the art'', we include those that best relate to our work. Readers interested in other models with different particularities can explore the references contained within our reference list.}
  An optimal design of a vaccination program requires, apart from financial and logistical considerations, subtle results in epidemiology that are currently not available in the literature. 
 Several works focus on the study of epidemiological models with pulse vaccination.

In 2002, Lu {\it et al.}~\cite{LuChiChen2002} explored constant and pulse vaccination strategies in an SIR model with vertical and horizontal transmission. Authors showed that the effectiveness of the vaccination strategy depended on the interval between boosters. They   observed that the high susceptibility in the offspring of infected parents accelerates stabilisation, emphasising the importance of parental health.  Numerical simulations supported their findings.

Wang \cite{Wang2015} studied the periodic oscillation of seasonally forced epidemiological models with pulse vaccination. Using Mawhin's coincidence degree method, the author confirmed the existence of positive periodic solutions for these SIR models with pulse vaccination. The effectiveness of this vaccination strategy was supported by numerical simulations. For further investigation using Mawhin's degree of coincidence method, we address the reader to \cite{Feltrin2018, Wang2020, Jodar2008, ZuWang2015} for a more in-depth understanding.

Meng and Chen \cite{Meng2008} analyzed a SIR epidemic model with vertical and horizontal disease transmission. The authors also revealed that under some conditions, the system is {\it permanent}. Moreover, for $\mathcal{R}_0 > 1$ ({\it basic reproduction number}), the system under consideration exhibits positive periodic solutions. 

 Using data from Thailand, authors of  \cite{Kanchanarat2023} concluded that high vaccination coverage is not enough to eradicate measles without an optimised schedule for vaccine shots.

In 2022, authors of \cite{CarvalhoRodrigues2022} investigated a periodically-forced dynamical system inspired by the  SIR model. They provided a rigorous proof of the existence of observable chaos, expressed as persistent {\it strange attractors} on subsets of parameters for $\mathcal{R}_0 < 1$, where $\mathcal{R}_0$ stands for the basic reproduction number. Their results are in line with the empirical belief that intense seasonality can induce chaotic behavior in biological systems. 

 In 2023, Ibrahim and D\'enes \cite{IbrahimDenes2023} developed a seasonal mathematical model to study the transmission of measles, applied to real data of Pakistan. The authors found that measles can become endemic and repeat annually when $\mathcal{R}_0 > 1$. The study showed that increasing vaccination coverage and effectiveness is crucial to reducing transmission and mitigating future outbreaks.

 In a recent study by Guan {\it et al.}   \cite{Guan2024}, an impulsive model was applied to rubella infection data in China. The study evaluated the effectiveness of the impulse vaccination strategy in eradicating rubella, incorporating environmental and genetic factors. Their findings highlighted that pulse vaccination can successfully eradicate rubella under favourable conditions. 


\subsection*{Achievements}
The present work provides insights into the interplay between seasonality, impulsive differential equations and the presence of horseshoes in periodically-forced epidemic models. The main goals of this article are the rigorous proof of the following assertions:

\begin{enumerate}
\item  in the absence of seasonality, the model goes through five scenarios when varying the parameters of the period $T$ of the vaccine, and the proportion of {\it Susceptible} individuals vaccinated $p$;
\smallskip
\item in the absence of seasonality, we design an optimum vaccination program as a function of the period $T$ of pulse vaccination and the proportion $p$ of {\it Susceptible} individuals that need to be vaccinated to control the disease;
\smallskip
\item under the presence of seasonality in the rate transmission rate of the disease, the system may behave chaotically and exhibits topological horseshoes.
\end{enumerate}

\noindent The bifurcations between the different scenarios have been identified and explored. All the results  are illustrated with numerical simulations.  

\subsection*{Structure}


 In this paper, we analyze a modified SIR model to study the impact of the pulse vaccination strategy with and without seasonality. Section \ref{s:preliminaries} provides fundamental insights into impulsive differential equations and the important concepts of periodic solutions and stability. Section \ref{model_model} introduces the model (with and without vaccination and with and without seasonality) and clarifies the role of the  hypotheses. Also, in this section, we present our two main results. From Section \ref{PREP_SECT} to Section \ref{Section_of_stability}, we analyze the model from a mathematical point of view: we find the fixed points of the associated stroboscopic maps,  compute the {\it basic reproduction number}, and evaluate the local stability of the disease-free periodic solution. From Section \ref{section: item (2)} to Section \ref{SA_lab}, we prove all items of the main results. In Section \ref{s:numerics}, we provide some numerical simulations that support the theoretical results and, finally, in Section \ref{s:section}, we relate our findings with others in the literature.

\section{Preliminaries}
\label{s:preliminaries}
For the sake of   self-containedness of the paper, we present the basic definitions and notation of the theory of impulsive dynamical systems we need. We also include some fundamental results which are necessary for understanding the theory. This information can  be found in \cite{Dishliev, Lakshmikantham_livro} and Chapter 1 of \cite{Agarwal_livro}. 

\subsection{Instantaneous impulsive differential equations}\label{object}

An impulsive differential equation is given by an ordinary differential equation coupled with a discrete map defining the ``jump'' condition. The law of evolution of the process is described by the differential equation
 $$
 \frac{dx}{dt}=f(t,x)
 $$
  where $t\in \RR$, $x\in \Lambda \subset \RR^n$ and $f: \RR  \times \Lambda \rightarrow \Lambda $ is $C^1$.
The instantaneous impulse at time $t$ is defined by the jump map $J(t,x): \RR \times \Lambda \rightarrow \Lambda$ given by
$$
(t,x)\mapsto x+J(t,x).
$$

\noindent Throughout this article, we focus on the {\it Instantaneous impulsive equation}:
\begin{eqnarray}
\label{prel1}
\frac{dx}{dt}&=&f(t,x), \qquad t\neq T_k, \\ 
\nonumber \Delta x(T_k), &=& J_k(x),  \qquad  \,    k\in \NN_0,   \\
\nonumber x(0)&=&x_0 ,   
\end{eqnarray}

\medskip

\noindent where $x\in \Lambda \subset \RR$,  the impulse is fixed at the sequence $(T_k)_{k\in \NN_0}$ such that $T_0=0$ and
$$\forall k\in \NN \quad T_k < T_{k+1} \qquad \text{and }\qquad \lim_{k\in \NN}T_k=+\infty.$$
\noindent The  instantaneous ``jump'' $\Delta x(T_k)=J(T_k,x)\equiv  J_k(x)$ of \eqref{prel1} is defined by
$$
\Delta x(T_k):= \lim_{t\rightarrow T_k^+}\varphi(t,x)-  \lim_{t\rightarrow T_k^-}\varphi(t,x). 
$$
 
\noindent For $k\in \NN$ and $t  \in [T_k,T _{k+1})$,  $\varphi(t,x)$ is a solution of $\dpt \frac{dx}{dt}=f(t,x)$; for $t= T_k$, $\varphi$  satisfies
$$\lim_{t\rightarrow T_k^+}\varphi(t,x)= \lim_{t\rightarrow T_k^-}\varphi(t,x) + J_k\left(\lim_{t\rightarrow T_k^-}\varphi(t,x)\right),$$
  where $x\in \Lambda$.  The next result concerns the existence of a unique solution for  \eqref{prel1}.

\medskip

\begin{prop}[\cite{Lakshmikantham_livro, Bainov_livro}, adapted]
Let the function $f \colon \RR\times \Lambda \rightarrow \RR^n$ be continuous in the sets $[T_k, T_{k+1}[ \times \Lambda$, where $k\in \NN$. For each $k \in \NN$ and $x \in \Lambda$, suppose there exists (and is finite) the limit of $f(t,y)$ as $(t,y)\rightarrow (T_k,x)$, where $ t >T_k$. Then, for each $(t_0, x_0)\in \RR\times \Lambda$ there exist $\beta>t_0$ and a solution $\varphi\colon \, ]t_0, \beta[\rightarrow \RR^n$ of the  IVP \eqref{prel1}. If $f$ is $C^1$ with respect to $x$ in $\RR\times \Lambda$, then the solution is unique.
\end{prop}

 The following result imposes conditions where  the solution $\varphi$ may be {\it extendable}.

\begin{prop}[\cite{Lakshmikantham_livro,Bainov_livro}, adapted]
\label{Theorem2.1}
Let the function $f \colon \RR\times \Lambda \rightarrow\RR^n$ be $C^1$ in the sets $[T_k, T_{k+1}[ \times \Lambda$, where $k\in \NN$. For each $k \in \NN$ and $x \in \Lambda$, suppose there exists (and is finite) the limit of  $f(t,y)$ as $(t,y)\rightarrow (T_k,x)$, where $ t >T_k$. If \,$\varphi \colon  ]\alpha, \beta[ \rightarrow \RR^n$ is a solution of \eqref{prel1}, then the solution is extendable to the right of $\beta$  if and only if
$$\dpt \lim_{t\rightarrow \beta^-}\varphi(t,x)=\eta ,$$

\smallskip

\noindent and one of the following conditions holds:

\smallskip

\begin{enumerate}
\item $\beta\neq T_k$, for any $k\in \NN_0$ and $\eta \in \Lambda$;

\medskip

\item $\beta = T_k$, for some $k\in \NN_0$ and $\eta + J(T_k, \eta) \in \Lambda$.
\end{enumerate}
\end{prop}

 Under the conditions of Proposition \ref{Theorem2.1},  for each $(t_0 , x_0 ) \in \RR \times \Lambda$, there exists a
unique solution $\varphi(t, x_0)$ of   \eqref{prel1} defined in $\RR$ (\cite{Lakshmikantham_livro, Bainov_livro}) which can be written as
 \begin{equation*}
\varphi(t, x_0)= \left\{
\begin{array}{l}
\dpt x_0 +\int_{t_0}^t f(s, \varphi(s, x_0)) \text{ds} + \sum_{t_0\leq T_k<t}J(\varphi(T_k, x_0)) \quad {\text{for}}\quad t\geq t_0,\\ 
\dpt x_0 +\int_{t_0}^t f(s, \varphi(s, x_0)) \text{ds} + \sum_{t\leq T_k<t_0}J(\varphi(T_k, x_0))  \quad {\text{for}}\quad t\leq t_0.\\ 
\end{array}
\right.
\end{equation*}

\smallskip

\begin{defn}
We say that $\mathcal{K} \subset (\RR^+_0)^2$ is a {\it positively flow-invariant set} for \eqref{prel1} if for all $x \in \mathcal{K}$, the trajectory of $\varphi(t, x)$ is  contained in $\mathcal{K}$ for $t \geq 0$.
\end{defn}
 
 For a solution of \eqref{prel1} passing through $x\in \RR^n$, the set of its accumulation points, as $t$ goes to $+\infty$, is the {\it $\omega$-limit set} of $x$. More formally, if $\bar{A}$ is the topological closure of $A\subset \RR^n$, then: 
\begin{defn} If $x\in \RR^n$, then the {\it $\omega$-limit} of $x$ is
 $$
\omega(x)=\bigcap_{T=0}^{+\infty} \overline{\left(\bigcup_{t>T}\varphi(t, x)\right)}.
$$ 
\end{defn}

  It is well known that $\omega(x)$  is closed and {\it flow-invariant}, and if the  trajectory of $x$ is contained in a compact set, then 
$\omega(x)$ is non-empty. If $E$ is an invariant set of \eqref{prel1}, we say that $E$ is a {\it global attractor} if $\omega(x) \subset E$, for Lebesgue almost all points $x$ in   $\mathbb{R}^n$.

\subsection{Periodic solutions and stability}
\label{ss: definitions}
The following definitions have been adapted from \cite{Milev1990, SimeonovBainov1988, HirschSmale1974}.
For $T>0$, we say that $\varphi$ is a {\it $T$-periodic} solution of \eqref{prel1} if and only if there exists $x_0\in \Lambda$ such that 
\begin{equation}
\label{periodic_def}
\forall  t\in \RR \quad \varphi (t, x_0)= \varphi (t+T, x_0).
\end{equation}

\noindent  We disregard constant solutions and we consider  the smallest positive value $T$ for which \eqref{periodic_def} holds. Let $x_0 \in \Lambda$ be such that $\varphi(t, x_0 )$ is a  {\it $T$-periodic} solution of \eqref{prel1}. We say that $\varphi(t, x_0)$ is:
\smallskip
\begin{enumerate}
\item {\it stable} if, for any neighborhood $V$ of $x_0$, there is a neighborhood $W\subset V$ of $x_0$ such that for all $y_0 \in W$ and for all $k\in \NN$ we have $\varphi(kT,y_0 )\in V$;
\smallskip
\item {\it asymptotically stable} if, it is stable and there exists a neighborhood
$V$ of $x_0$ such that for any $y_0 \in V$ and $\dpt \lim_{k\rightarrow +\infty} \varphi(kT, y_0 )= x_0$;
\smallskip
\item {\it unstable} if it is not stable.
\end{enumerate} 

\section{Model} \label{model_model}

Similarly to the classical SIR model \cite{Dietz1976}, the population in the model under consideration is divided into three subpopulations:

\smallskip

\begin{itemize}
\item {\it Susceptibles}: individuals that are currently not infected but can contract the disease;  

\smallskip

\item  {\it Infectious}: individuals who are currently infected and can actively transmit the disease to a susceptible individual until their recovery;

\smallskip

\item {\it Recovered}: individuals who currently can neither be infected nor infect susceptible individuals. This comprises individuals with permanent immunity because they have recovered from a recent infection or have been vaccinated. 
\end{itemize}

\medskip

  Let $S$, $I$, and $R$ denote the proportion of individuals within the compartment of {\it Susceptible}, {\it Infectious}, and {\it Recovered} individuals (within the whole population).
\noindent  Susceptible individuals are those who have never had contact with the disease. Once they have  contact with the disease, they become {\it Infectious}. They  are ``transferred'' to the {\it Recovered} class if they do not die. Those who recover from the disease get lifelong immunity. We add effective pulse vaccination to a proportion $p\in [0,1]$ of the Susceptible Individuals, providing lifelong immunity.

We propose the following nonlinear impulsive differential equation in $S$, $I$, and $R$ (depending on time $t\in \RR_0^+$):

\begin{equation}
\label{modeloSIR}
\begin{array}{lcl}
\dot{X} = \mathcal{F}_\gamma(X) \quad \Leftrightarrow \quad
\begin{cases}
\medskip
&\dot{S} = S(A-S) - \beta_{\gamma} {(t)} I S  \\
\medskip
&\dot{I} =\beta_{\gamma}{(t)}IS - \left(\mu +d+g\right) I,   \qquad \qquad \quad t \neq nT  \\
\bigskip
&\dot{R} = gI - \mu R \\
\medskip
&S(nT) = (1-p) S(nT^-)  \\
\medskip
&I(nT) = I(nT^-)   \qquad \qquad \qquad \quad \quad 
\\
\medskip
&R(nT) = R(nT^-) + p S(nT^-)
\end{cases}
\end{array}
\end{equation}

\smallskip

\noindent where $n \in \NN$, $p\in [0,1]$,

\smallskip

$$
\begin{array}{rcl}
X(t) &=& \left( S(t), I(t), R(t) \right), \\
\\
X(t_0) &\coloneqq& X_0 \,\,\, \text{is the initial condition  (in general $t_0=0$)}, \\
\\
\dot{X} &=& \left(\dot{S}, \dot{I}, \dot{R}\right) \,\,\, = \,\,\, \dpt \left(\frac{\mathrm{d}S}{\mathrm{d}t},\frac{\mathrm{d}I}{\mathrm{d}t},\frac{\mathrm{d}R}{\mathrm{d}t}\right), \\
\\
\beta_\gamma(t) &=& \beta_0 \left(1+\gamma \Psi (\omega t)\right), \\
\\
 X(nT^-) &=& \displaystyle \lim_{t \,\rightarrow \,nT^-} X(t),
\end{array}
$$
\medbreak

 \noindent $A \in (0,1]$, $\gamma \geq 0$, $\omega>0$ and $\Psi$ is   periodic. The number $S(nT^-)=\displaystyle \lim_{t \,\rightarrow \,nT^-} S(t)$ corresponds to the {\it Susceptibles} at the instant immediately before being vaccinated for the  $n^{th}$ time. The model without vaccination corresponds to $p=0$. Figure \ref{boxes} represents the dynamics of~\eqref{modeloSIR}.

\usetikzlibrary{arrows,positioning}
\begin{center}
\begin{figure}[ht!]
\begin{tikzpicture}
[
auto,
>=latex',
every node/.append style={align=center},
int/.style={draw, minimum size=1.75cm}
]

    \node [fill=lightgray,int] (S)             {\Huge $S$};
    \node [fill=lightgray,int, right=5cm] (I) {\Huge $I$};
    \node [fill=lightgray,int, right=11cm] (R) {\Huge $R$};
    
    \node [below=of S] (s) {} ;
    \node [below=of I] (i) {} ;
    \node [below=of R] (r) {};
    
    \node [above=of S] (os) {};
    \node [above=of I] (oi) {};
    \node [above=of R] (or) {};
    
    \coordinate[right=of I] (out);
    \path[->, auto=false,line width=0.35mm]
    			(S) edge node {$\beta_\gamma(t) I S$ \\[.6em]} (I)
                          (I) edge node {$gI$       \\[.6em] } (R) ;

    \path[->, auto=false,line width=0.35mm]
			(S) edge  [out=-120, in=-60] node[below] {$S(A-S)$ \\ [0.2em] {\it Logistic growth}} (S)
    			(I) (5.88,-0.87cm) edge [] node[right]{$\sigma I$} (5.88,-2cm) (i) 
    			(R) (11.88,-0.87cm) edge [] node[right]{$\mu R$} (11.88,-2cm) (r) ;

    \path[-, auto=false,line width=0.35mm]
    			(S) (0,0.87cm) edge [] node[right]{} (0,2cm) (os)
			(os) (0,2cm) edge [] node[above]{$p S(nT^-)$} (11.88,2cm) (or)
			(or) (11.88,2cm) edge [->] node[below]{} (11.88,0.87cm) (R) ;
			
     
\end{tikzpicture}
\caption{\small Schematic diagram of model \eqref{modeloSIR}. Boxes represent the subpopulations  $S,I$, and $R$, and arrows indicate the flow between the compartments.}
\label{boxes}
\end{figure}
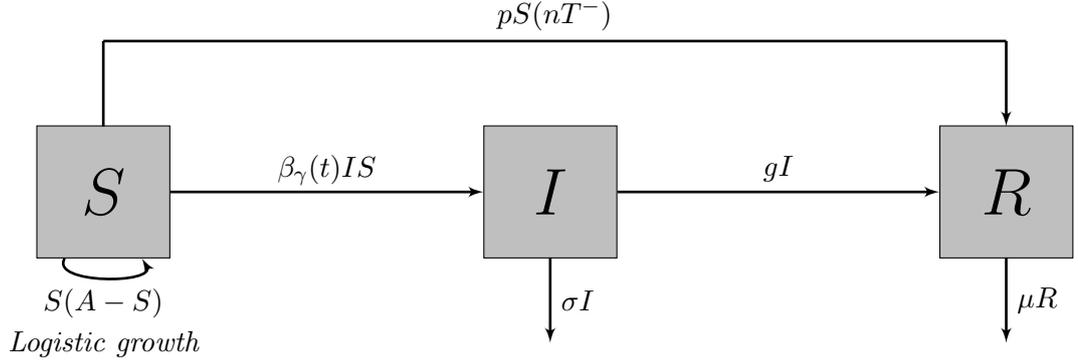
\end{center}

\subsection{Description of the parameters} We can describe the parameters of   \eqref{modeloSIR} as: \\

\begin{description}
\item[$A$] the carrying capacity of the {\it Susceptible} for $\beta_0=0$, {\it i.e.} in the absence of disease; 

\smallskip

\item[$\gamma$]  the amplitude of the seasonal variation that oscillates between $\dpt  \beta_0 \Big(1+\gamma \min_{t\in [0,\tau]} \Psi(t) \Big)>0$ in the low season, and $\dpt \beta_0 \Big(1+\gamma \max_{t\in [0,\tau]} \Psi(t) \Big)$ in the high season, for some $\tau>0$; 

\smallskip

\item[$\Psi(\omega t)$] the effects of periodic seasonality over the time $t$ with frequency $\omega>0$; 

\smallskip

\item[$\beta_0$]  the disease transmission rate in the absence of seasonality (when $\gamma=0$). The parameter $\gamma$ ``measures the deformation'' of the transmission rate;

\smallskip

\item[$\mu$] the natural death rate of infected and recovered individuals;

\smallskip

\item[$d$] the death rate of infected individuals due to the disease;

\smallskip

\item[$p$] the proportion of the {\it Susceptibles} periodically vaccinated; 

\smallskip

\item[$g$] the cure rate;

\smallskip

\item[$T$] the positive period at which a proportion $p$ of the {\it Susceptible} individuals is vaccinated.
\end{description}
\bigbreak
\noindent In order to simplify the notation, we denote $\mu + d$ by $\sigma$.

\subsection{Hypotheses and motivation}

Regarding   \eqref{modeloSIR}, we assume:

\smallskip

\begin{itemize}
\item[{\bf (C1)}] All parameters are non-negative;

\smallskip
 
\item[{\bf (C2)}] For all $t \in \RR^+_0$, $S(t) \leq A$;

\smallskip
 
\item[{\bf (C3)}]  $S(t)$, $I(t)$, and $R(t)$ are proportions over the whole population. In particular, we have $S(t) + I(t) + R(t) = 1$, for all $t\in\RR_0^+$;

\smallskip
\item[{\bf (C4)}] For $\tau > 0$ and $\gamma > 0$, the map $\Psi: \RR \rightarrow \RR^+$ is  $\tau$-periodic, $\dpt \dfrac{1}{\tau} \int_0^{\tau} \beta_{\gamma}(t) \, \mathrm{d}t > 0$ and has (at least) two nondegenerate critical points\footnote{This corresponds to a generic periodically-forced perturbation.}.
\end{itemize}

\medskip
 From  {\bf (C2)} and  {\bf (C3)}, since $A$ is the carrying capacity of $S$, then $A \in (0,1]$. The phase space of \eqref{modeloSIR} is a subset of $(\mathbb{R}_0^+)^3$, endowed with the usual Euclidean distance (denoted by ${dist}$) and the set of parameters is described as follows (for $\varepsilon>0$ small):

$$  {{\Omega} \subset \left\{ (A,\beta_0,p,\sigma,g) \in  (\RR_0^+)^5 \right\}, \qquad T \in \RR_0^+, \quad \gamma\in [0,\varepsilon] \quad \text{and}\quad  \omega \in \mathbb{R}^+}.$$

\medskip
 \begin{rem}
System \eqref{modeloSIR} has been motivated by the classical SIR model \cite{KermackMcKendrick1932} with the following modifications:

\begin{itemize}
\item We analyze the logistic growth of the {\it Susceptible} individuals as a result of crowding and natural competition for resources \cite{CarvalhoRodrigues2022, Li2017, ZhangChen1999} instead of assuming linear or exponential growth;

\smallskip

\item  The model includes a pulse-vaccination strategy to fight the disease where the {\it Susceptible} individuals are $T$-periodically vaccinated \cite{Shulgin1998,Agur1993}.
\end{itemize}
 \end{rem}
\smallskip
 The first two equations of (\ref{modeloSIR}), $\dot{S}$ and $\dot{I}$, are independent of $R$. Therefore we can reduce (\ref{modeloSIR}) to
 
\begin{equation}
\label{modelo2}
\begin{array}{lcl}
\dot{x} = f_{\gamma}(x) \quad \Leftrightarrow \quad
\begin{cases}
\medskip
&\dot{S} = S(A-S) - \beta_{\gamma}  {(t)}I S  \\
\bigskip
&\dot{I} = \beta_{\gamma}  {(t)}IS - \left(\sigma+g\right) I , \, \qquad \qquad \qquad t \neq nT  \\
\medskip
 &S(nT) = (1-p) S(nT^-)  \\
\medskip
&I(nT) = I(nT^-)  
\end{cases}
\end{array}
\end{equation}

\noindent with $x=(S,I)\in (\RR_0^+)^2$.

\subsection{No vaccination and no seasonality}
 If we do not consider neither vaccination nor seasonality in  \eqref{modelo2} ($\Leftrightarrow \,\, p = 0$ and $\gamma=0$), then   (\ref{modelo2}) can be rewritten as
 
\begin{equation}
\label{no_vaccine}
\begin{cases}
\medskip
&\dot{S} =  S(A-S) - \beta_0 I S \\
&\dot{I} =  \beta_0IS - \left(\sigma+g\right) I 
\end{cases}
\end{equation}

\medskip

\noindent and the {\it basic reproduction number} $\mathcal{R}_0$ can be explicitly computed as:

\begin{equation}
\label{R0}
\begin{array}{lcl}
\mathcal{R}_0 =\dpt \lim_{\tau \rightarrow +\infty} \frac{1}{\tau} \int_0^{\tau}  \dfrac{A\beta_{\gamma}(t)}{\sigma + g} \, \mathrm{d}t \overset{\gamma = 0}{=}   \dfrac{A \beta_0}{\sigma + g} > 0.\\
\end{array}
\end{equation}
 
\medskip

\noindent The {\it basic reproduction number} $\mathcal{R}_0$ is  an epidemiological measure that indicates the average number of new infections caused by a single infected individual in a completely susceptible population (\cite{CarvalhoRodrigues2023,  Li2011}).
\medbreak
If $A$, $\beta_0$, $\sigma$, and $g$  are such that   $\mathcal{R}_0< 1$, the dynamics of \eqref{no_vaccine} is quite simple: all solutions with initial condition $S_0>0$ converge to $(S, I)=(A,0)$, as proved in  \cite{CarvalhoRodrigues2023}. Otherwise, if $\mathcal{R}_0>1$,  then all solutions with initial condition $S_0, I_0>0$  converge to the endemic equilibrium 

\begin{equation}
\label{endemic_p=0}
(S, I)=\left(\frac{\sigma+g}{\beta_0},\frac{A\beta_0-(\sigma+g)}{\beta_0^2}\right)=\left(\frac{A}{\mathcal{R}_0}, \frac{A}{\beta_0}\left(1-\frac{1}{\mathcal{R}_0}\right)\right).
\end{equation}

\medskip

From now on, we   analyze the model with pulse vaccination and no seasonality, {\it i.e.} $p \neq 0$ and $\gamma=0$.  

\subsection{Pulse vaccination and no seasonality}

 In the absence of seasonality ($\gamma = 0$), the model (\ref{modelo2}) can be recast into the form
\begin{equation}
\label{modelo3}
\begin{array}{lcl}
\dot{x} = f_0(x) \quad \Leftrightarrow \quad
\begin{cases}
\medskip
&\dot{S} = S(A-S) - \beta_0 I S  \\
\bigskip
&\dot{I} = \beta_0 IS - \left(\sigma+g\right) I  , \qquad \qquad \qquad  t \neq nT  \\
\medskip
&S(nT) = (1-p) S(nT^-)  \\
\medskip
&I(nT) = I(nT^-)  \qquad \qquad \qquad \qquad 
\end{cases}
\end{array}
\end{equation}
\noindent whose flow is given by
\begin{equation}
\label{flow2D}
\varphi_0 \left( t, (S_0, I_0) \right)=[S( \left( t, (S_0, I_0) \right), I  \left( t, (S_0, I_0) \right)] \,, \quad t \in \RR_0^+ \,, \quad (S_0, I_0) \in \left( \RR_0^+ \right)^2 .\\
\end{equation}

\smallskip

We denote by $S_c$ the {\it epidemic critical threshold} associated to \eqref{modelo3}: 

\begin{equation}
\label{S_c(def)}
S_c= \frac{\sigma+g}{\beta_0} > 0 . \\ 
\end{equation} 

\medskip

\noindent The meaning of $S_c$ will be explained immediately after Lemma \ref{infected_decreases}.  Before stating the main results, we provide two definitions  adapted from \cite{TangChen2001}:

\begin{defn}
System \eqref{modelo3} is said to be:

\medskip

\begin{enumerate}
\item  {\it uniformly persistent} \,if there are  constants $c_1, c_2,T_0>0$ such that for all solutions $\left( S(t), I(t) \right)$ with initial conditions $S_0 > 0$ and $I_0 > 0$, we have
$$c_1 \leq S(t) \quad \text{and} \quad c_2 \leq I(t),$$ for all $t \geq T_0$;

\medskip

\item  {\it permanent} \,if it is uniformly persistent and bounded, that is, there are   constants $c_1, c_2$, $C_1, C_2,T_0>0$ such that for all solutions $\left( S(t),I(t) \right)$ with initial conditions
$S_0 > 0$ and $I_0 > 0$, we have
$$ c_1 \leq S(t) \leq C_1 \quad \text{and} \quad c_2 \leq I(t) \leq C_2, $$ for all $t \geq T_0$.
\end{enumerate}
\end{defn}

The main result  of this article provides a complete description of the dynamics of \eqref{modelo3} through a bifurcation diagram. We also exhibit an explicit expression for the {\it basic reproduction number} $\mathcal{R}_p$ for the system with impulsive vaccination. 

\begin{maintheorem}
\label{thA}
For $\gamma=0$ and $A>S_c$, in the bifurcation diagram $(T, p)\in (\RR_0^+)^2$ associated to \eqref{modelo3}, we may define the maps $p_1,p_2: \RR_0^+ \rightarrow [0,1]$ given by

\smallskip

$$
p_1(T)=1 -e^{-AT} \quad \text{and}\quad p_2(T)=1 -e^{-(A-S_c)T}
$$  
 such that:\\

\begin{enumerate}
\item if $p=1$, then the $\omega$-limit of all solutions of \eqref{modelo3} is the disease-free   periodic solution associated to $(S,I)=(0,0)$;\\
\item if $p\in (p_1(T), 1)$, then the $\omega$-limit of  all  solutions of \eqref{modelo3} is the disease-free  periodic solution associated to $(S,I)=(0,0)$;\\
\item if $p\in (p_2(T), p_1(T))$, then the $\omega$-limit of  all  solutions of \eqref{modelo3} with initial condition  $S_0>0$ is a disease-free  (non-trivial) periodic solution $(\mathcal{S},0)$\footnote{In Section \ref{Section_of_global_stability}, we provide an explicit expression for the periodic solution  $(\mathcal{S},0)$; all trajectories converge to $(\mathcal{S},0)$ in the   topology of pointwise convergence.};\\
\item if $p\in (0, p_2(T))$, then system \eqref{modelo3} is permanent and the $\omega$-limit of  all  solutions of \eqref{modelo3} with initial condition   $S_0, I_0>0$  is an endemic periodic solution $(\mathcal{S},\mathcal{I})$;\\
\item if $p=0$, then the $\omega$-limit of   all  solutions of \eqref{modelo3} with initial condition   $S_0, I_0>0$ is the endemic equilibrium defined in \eqref{endemic_p=0};\\
\item The curves $p_1$ and $p_2$ correspond to saddle-node and transcritical bifurcations, respectively; 
\item The basic reproduction number associated to \eqref{modelo3} is $\dpt\mathcal{R}_p= \mathcal{R}_0 \left[ \dfrac{\ln{(1-p)}}{AT} + 1 \right]$ and $R_p>1$ if and only if $p<p_2$.\\

\end{enumerate}
\end{maintheorem}

 The  scenarios {\Large \ding{172}}--{\Large \ding{176}} of Theorem \ref{thA} are represented in  the bifurcation diagram $(T,p)$ of Figures \ref{Bif_Diag_ThA} and \ref{imagem_94}. Its proof   is performed in several sections throughout the  present article  and  its location is indicated in Table \ref{table1}.

\begin{figure*}[ht!]
\includegraphics[width=0.82 \textwidth]{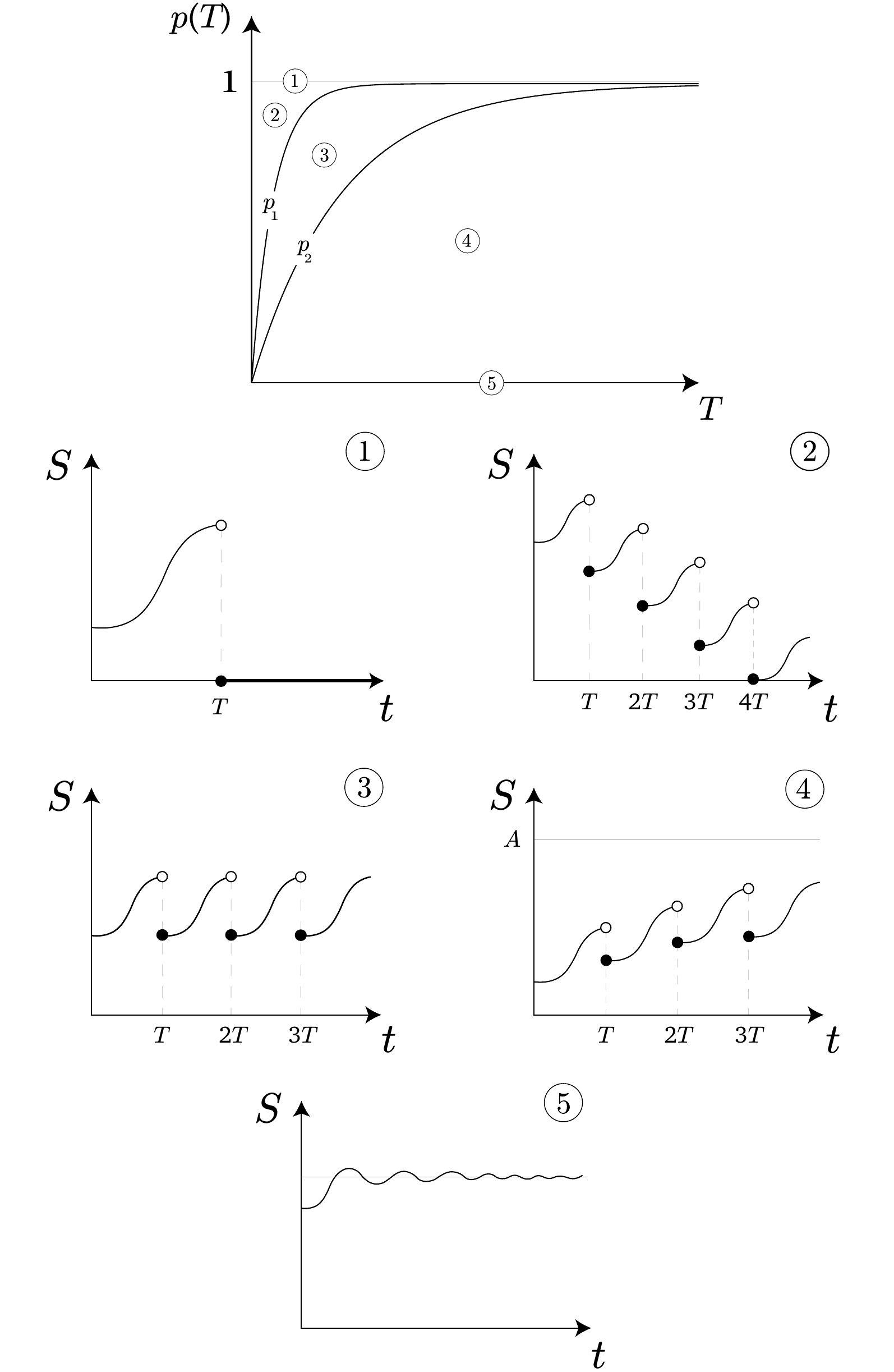}
\caption{\small Bifurcation diagram $(T,p)$ associated to \eqref{modelo3}. In {\Large \ding{172}}, all solutions converge to $(0,0)$;
in {\Large \ding{173}}, all solutions converge to $(0,0)$;
in {\Large \ding{174}},  all  solutions   with initial condition   $S_0>0$ tend towards the disease-free (non-trivial) periodic solution $(\mathcal{S},0)$;
in {\Large \ding{175}},  all  solutions   with initial condition   $S_0, I_0>0$ tend towards the endemic periodic solution $(\mathcal{S},\mathcal{I})$;
in {\Large \ding{176}},  all  solutions   with initial condition   $S_0, I_0>0$ tend towards the endemic equilibrium $(\mathcal{S},\mathcal{I})$. Compare with the numerical simulation of these five scenarios in Figure \ref{art_}.}
\label{Bif_Diag_ThA}
\end{figure*}

\begin{figure*}[ht!]
\includegraphics[width=1.08 \textwidth]{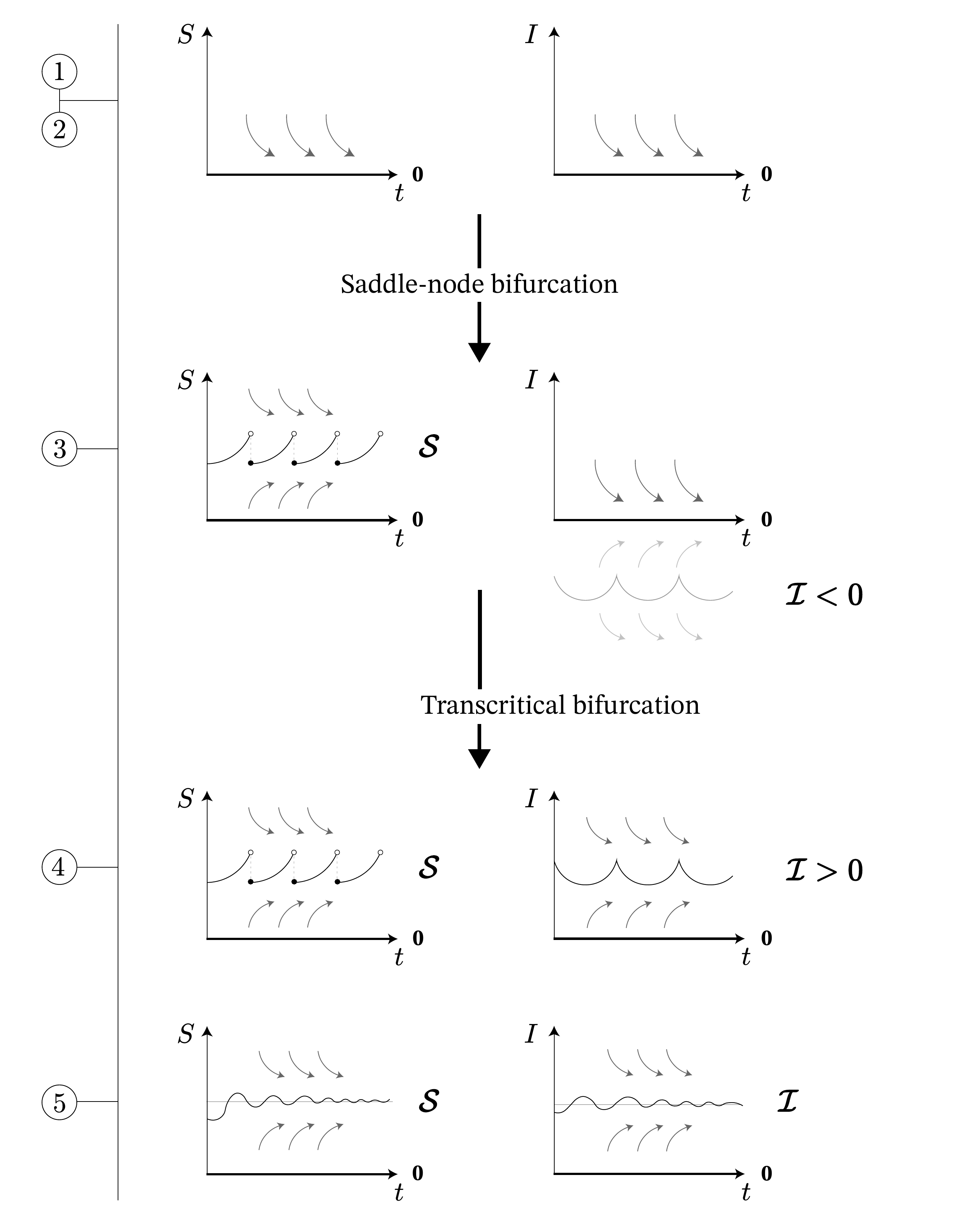}
\caption{\small  Bifurcations associated to the diagram $(T,p)$ associated to \eqref{modelo3}.
From {\Large \ding{173}} to  {\Large \ding{174}}: saddle-node associated to the solution $(0,0)$;
From {\Large \ding{174}} to  {\Large \ding{175}}: transcritical bifurcation associated to   $(\mathcal{S},0)$. For the  numbering {\Large \ding{172}} to  {\Large \ding{176}}, see Theorem \ref{thA} and Figure \ref{Bif_Diag_ThA}.}
\label{imagem_94}
\end{figure*}

\begin{table}[h!]
\small 
\centering
\renewcommand{\arraystretch}{1.4}
\begin{tabular}{ |c|l|  }
\hline 
Items of Theorem  \ref{thA}  & Reference  / Section  \\
\hline  
(1) &    Trivial   \\
(2) &    Section \ref{section: item (2)} \\
(3) &    Sections \ref{Section_of_stability} and \ref{Section_of_global_stability}  \\
(4) &    Section \ref{section_of_permanence}\\
(5) &    See reference \cite{CarvalhoRodrigues2023}   \\
(6) &     Section \ref{section:bif} \\
(7) &    Section \ref{section:Rp} \\
\hline
\end{tabular}
\medskip
\caption{\small Structure of  the proof of Theorem \ref{thA} and the location of the items throughout the present article.}
\label{table1}
\end{table}

\subsection{Pulse vaccination and seasonality}
 Seasonal variations may be captured by introducing periodically-perturbed terms into a deterministic differential equation \cite{Shulgin2000, Meng2008}. The periodically-perturbed term $\Psi(t)$ in $\beta_\gamma(t)$ may be seen as a natural periodic map over time with two global extrema (governing the high and lower seasons defined by weather conditions). 
Corollary \ref{cor1},  proved in Section \ref{proof_cor_1}, confirms the numerical results suggested by Choisy {\it et al.}~\cite{Choisy2006}: neglecting the effect of the amplitude of the seasonal transmission can lead to somewhat overoptimistic values of the optimal pulse period.
 
 \begin{cor}\label{cor1}
 For $\gamma, \omega>0$  small and $A>S_c$, the $T$-periodic solution $(\mathcal{S},0)$ is asymptotically stable if and only if $  p_2^{\text{\rm seas}}(T)<p<p_1(T)$, where 
 $$
p_1(T)=1 -e^{-AT} \quad \text{and} \quad {p_2^{\text{\rm seas}}(T)=1 - e^{- \left[(A - S_c)T + \gamma \dpt \int_0^T  \Psi (\omega t) \mathcal{S}(t) \, \mathrm{d}t \right]}}.
$$

 \end{cor}

 \begin{figure*}[ht!]
\includegraphics[width=0.65 \textwidth]{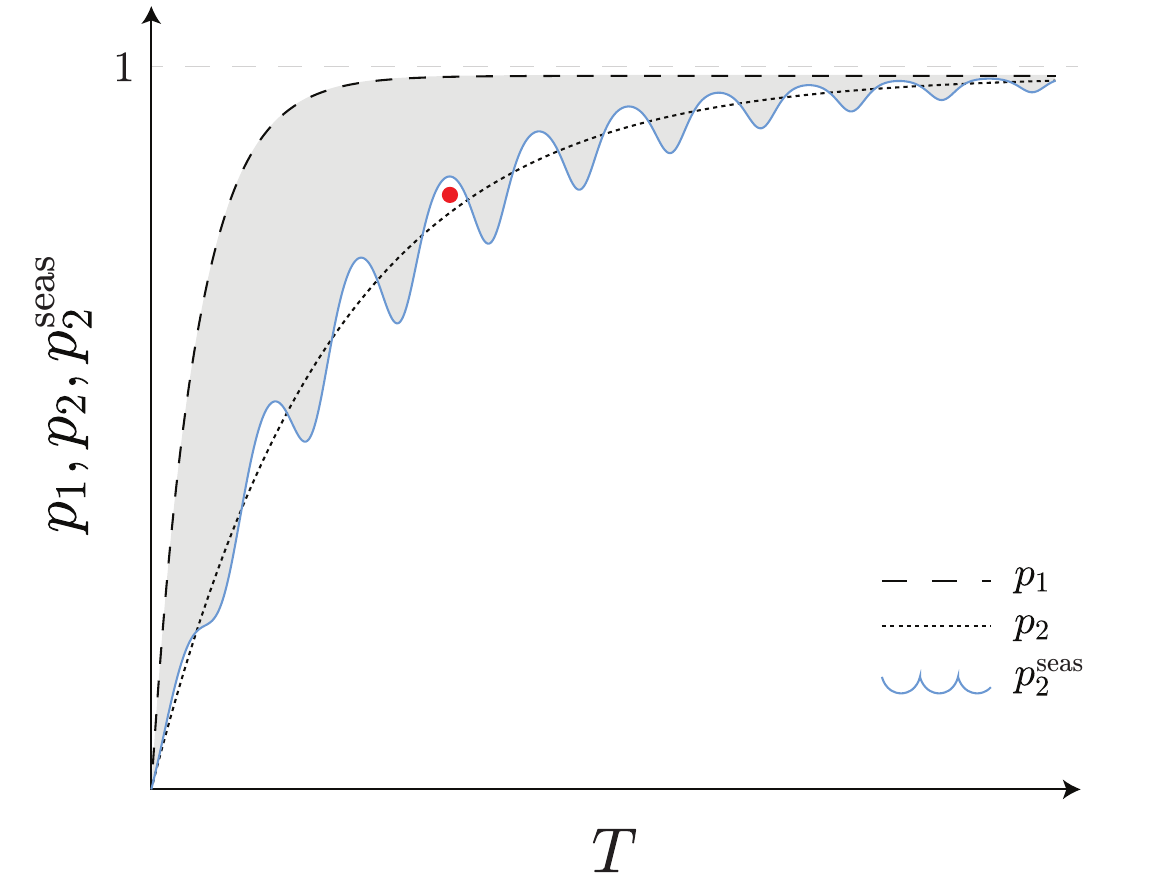}
\caption{\small Illustration of Corollary \ref{cor1}: the {\it $T$-periodic} solution $(\mathcal{S},0)$ is asymptotically stable if and only if  $p_2^{\text{seas}}(T)<p<p_1(T)$.  The red dot is a point  where $(\mathcal{S},0)$ is stable (without seasonality) and unstable (with seasonality).}
\label{p1p2SAZ}
\end{figure*} 
 
One aspect that contributes to the   complexity of \eqref{modelo2}  is the existence of chaos.
 Let $\mathcal{U}\subset \Omega$ such that  $0<p<p_2(T)$ ($\Rightarrow \mathcal{R}_p>1$). Under the condition that $T=k\tau/\omega$, $k \in \NN$,  there is an endemic {\it $T$-periodic} solution   for \eqref{modelo3} (see Item (4) of Theorem \ref{thA}).
 In what follows, we assume that $M$ is diffeomorphic to a  circloid. 
  
  \begin{defn}
  An embedding $F \colon M\rightarrow M$ is said to have a {\it horseshoe} if for some $N,n \in \NN$, the map $F^N$ has a uniformly hyperbolic invariant set $\Upsilon \subset M$ such that
$F^N|_\Upsilon$ is topologically conjugate to the full shift on $n$ symbols  $(\Sigma_n, \sigma)$ where $ \sigma$ is the usual shift operator. 
  \end{defn}
  A vector field possesses a  suspended horseshoe if the first return map to a cross-section has a horseshoe.  The existence of a horseshoe for the embedding $F$ is equivalent to the notion of topological chaos ($\Leftrightarrow$ $F$ has positive topological entropy).

\begin{maintheorem}
\label{th: mainB}
 For $0 < p < p_2^{\text{s\rm eas}}(T)$, $T=k\tau/\omega$, $k \in \NN$ and $\omega \in \RR^+$,  if $\gamma$ is sufficiently large, then the flow of ${f}_\gamma$ has a suspended topological horseshoe.
\end{maintheorem}
The shift dynamics obtained for (\ref{modelo2}) differ from  that of  \cite{Herrera2023}, as discussed in Section \ref{s:section}. The proof of Theorem \ref{th: mainB} is performed in Section \ref{SA_lab}.

\subsection{Biological consequences}

As suggested by Theorem \ref{thA}, the global eradication of an epidemic by means of pulse vaccination is always possible, provided the vaccination coverage is large enough.

  Based on experimental data, the World Health Organization recommends that the time between successive pulses should be as short as possible.
  For a specific vaccination coverage $p\in (0,1)$, there exists a pulse interval $(0, T_2)$ where $T_2= \dfrac{\left| \ln{(1-p)} \right|}{A-S_c}$   that ensures the effective implementation of this campaign ($\Leftrightarrow \mathcal{R}_p<1$); that is, the {\it $T$-periodic} administration of doses with $T \in (0, T_2)$ leads to the global eradication of the disease, and $T_2$ is the optimal time that determines the fastest eradication. For a specific vaccination coverage $p>0$, if $T>T_2$ then $\mathcal{R}_p>1$.
  
Adding seasonality to our model, the {\it epidemic critical threshold} $S_c$ also depends on $\beta_{\gamma}$. Corollary \ref{cor1} stresses that neglecting the effect of the amplitude of the seasonal transmission can lead to  overoptimistic   values of the optimal pulse period $T_2$.  It also stresses that the best moment for vaccination corresponds to the local minimizers of $\beta_\gamma$.  Theorem \ref{th: mainB} says that the number of {\it Susceptible} individuals and {\it Infectious} in the presence of seasonal variations could be unpredictable.  The distant future is practically inaccessible and may only be described in average, in probabilistic and ergodic terms.

\section{Preparatory section} \label{PREP_SECT}
In this section, we collect  lemmas needed for the proof of Theorems ~\ref{thA}, \ref{th: mainB} and Corollary \ref{cor1}. We start by proving the existence of a compact set where the dynamics lies.

\begin{lem} \label{region_}
The region defined by
$$
\mathcal{M} = \left\{ (S,I) \in  (\mathbb{R}_0^+)^2: \quad  0 \leq S \leq A, \quad 0 \leq S+I \leq \dfrac{A (\sigma + g + A)}{\sigma + g}, \quad S, I\geq 0 \right\}
$$
is positively flow-invariant for  \eqref{modelo3}.
\end{lem}

\begin{proof}
We can easily to check that $(\mathbb{R}_0^+)^2$ is {\it flow-invariant}. Now, we show that if $(S_0, I_0)\in \mathcal{M}$, then $\varphi_0 \left( t, (S_0, I_0) \right)$, $t \in \mathbb{R}_0^+$, is contained in $ \mathcal{M}$. Let us define
$$\nu(t) = S(t) + I(t)\geq 0$$
associated to the trajectory $\varphi_0 \left( t, (S_0, I_0) \right)$. Omitting the variables' dependence on $t$, one knows that

\begin{equation}
\label{proof1}
\begin{array}{lcl}
\dot{\nu} & = & \dot{S} + \dot{I} \\
\\
& = & S(A-S) - \beta_0 I S + \beta_0 I S - (\sigma + g) I
\nonumber \\
\\
& = & S(A-S) - (\sigma + g)I ,\\
\end{array}
\end{equation}
from which we deduce that

\begin{equation}
\label{proof2}
\begin{array}{lcl}
\dot{\nu} + (\sigma + g) \nu & = & S(A-S) - (\sigma + g)I + (\sigma + g)S + (\sigma + g)I
\nonumber \\
\\
& = & S(A-S) + (\sigma + g)S  \\ \\
& \leq  &   (\sigma + g + A) S.
\end{array}
\end{equation}
\bigbreak

If $\beta_0 = p = 0$, the the first component of equation \eqref{modelo3} would represent logistic growth. Thus, its solution would have an upper bound of $A$ (by {\bf (C2)}). This property is also verified for $\beta_0, p \in \RR^+$, even if $p$ is applied periodically  (note that $p \in [0,1]$). In particular, we have
\begin{equation*}
\label{proof3}
\begin{array}{lcl}
\dot{\nu} + (\sigma + g) \nu \leq (\sigma + g + A) A.
\end{array}
\end{equation*}

\medskip

The classical differential version of the Gronwall's inequality\footnote{If $a,b \in \RR$ and $u: \RR_0^+\rightarrow \RR_0^+$ is a $C^1$ map  such that  $u'\leq au+b$, then $u(t)\leq u(0) e^{at}+\frac{b}{a}(e^{at}-1)$.} says that for all $t \in \RR^+_0$, we have
$$
\nu(t)\leq \nu_0 e^{-(\sigma+g) t} - \frac{(\sigma + g + A) A}{(\sigma+g)}\left(e^{-(\sigma+g) t}-1\right),
$$
where $\nu(0) \coloneqq \nu_0=S(0)+I(0)\geq 0$. Taking the limit when $t\rightarrow +\infty$, we get:
\begin{eqnarray*}
0\leq \lim_{t\rightarrow +\infty} \nu(t)&\leq& \lim_{t\rightarrow +\infty}  \left[\nu_0 e^{-(\sigma+g) t} - \frac{(\sigma + g + A) A}{(\sigma+g)}\left(e^{-(\sigma+g) t}-1\right) \right]\\
&=&   \frac{(\sigma + g + A) A}{(\sigma+g)}. 
\end{eqnarray*}
Since $\dpt \lim_{t\rightarrow +\infty} \nu(t)= \dpt \lim_{t\rightarrow +\infty}\left( S(t) + I(t)\right)$, the result is proved. 
\end{proof}

\bigbreak 

\begin{lem} \label{infected_decreases}
Let $D$ be an open subinterval of $\RR^+_0$.  With respect to \eqref{modelo3}, the following assertions hold:
 
\smallskip
 
\begin{enumerate}
\item  $S(t) < S_c $ for all $t\in D$ if and only if  $I$ is decreasing in $D$;
\medskip
\item $S(t) = S_c $ for all $t\in D$ if and only if  $  I$ is constant;
\medskip
\item   $S(t) > S_c $ for all $t\in D$  if and only if    $I$ is increasing in $D$.
\end{enumerate}
 
\end{lem}

\bigbreak 
\begin{proof}
We just show (1); the proof of (2) and (3) follows from the same reasoning.  We know that the {\it Infectious} is decreasing if and only if $\dot{I}(t) < 0$, for $t\in D$. Indeed,

\begin{eqnarray}
\nonumber  \dot{I}(t) < 0 &\overset{\eqref{modelo3}}{\Leftrightarrow}& \beta_0 S(t) I(t) - \left(\sigma+g\right) I(t) < 0 \\
& \Leftrightarrow&S(t) < \dfrac{\sigma + g}{\beta_0} = S_c  ,
 \label{R0Sc}
\end{eqnarray} 
and the result is proved.
\end{proof}

\smallskip

\begin{rem}
\label{S_c remark}
 The constant $S_c$ defined in \eqref{S_c(def)} will be called the {\it epidemic critical value} and is the  threshold on the number of {\it Susceptible} individuals that defines whether the epidemic spreads or not.    
When $\gamma=0$, the most intriguing dynamical scenario is when $S(t)-S_c$ takes different signs in $D\subset \RR_0^+$ (See Figure \ref{S_I}).  
 \end{rem}
  \begin{figure*}[ht!]
\includegraphics[width=0.84 \textwidth]{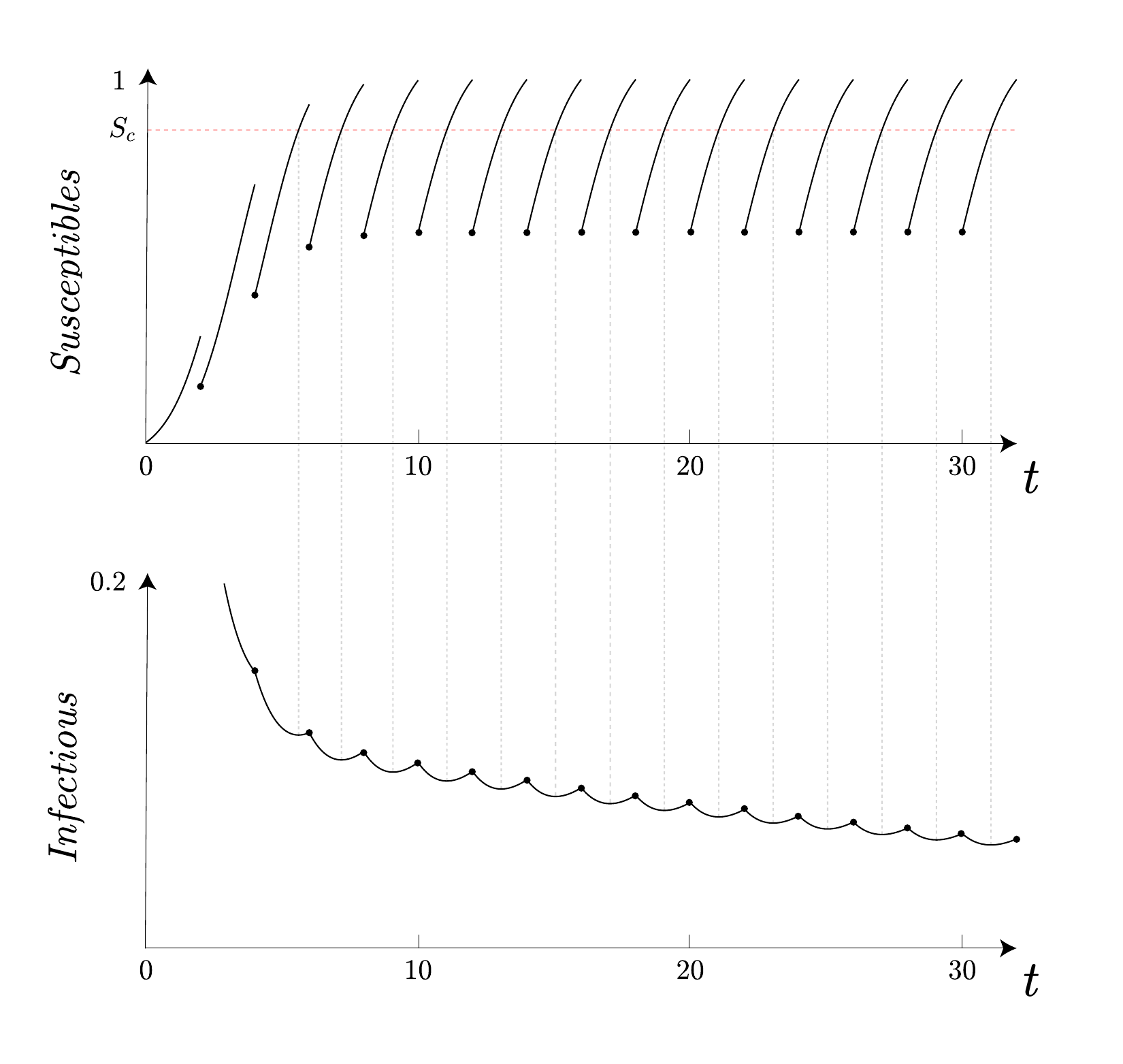}
\caption{\small Illustration of Lemma \ref{infected_decreases} for $D=\RR^+$: evolution of $S$ and $I$ as  $t\in \RR^+$ evolves.}
\label{S_I}
\end{figure*}

 \begin{lem}
 \label{A>S_c}
The following equivalence holds for system \eqref{modelo3}:
   $$A>S_c  \quad \Leftrightarrow \quad \mathcal{R}_0 > 1.$$
 \end{lem}
 
 \begin{proof}
 The proof follows from:
  $$A > S_c  \qquad \overset{\eqref{S_c(def)}}{\Leftrightarrow} \qquad A > \dfrac{\sigma + g}{\beta_0} 
   \qquad  \overset{\eqref{R0}}{\Leftrightarrow} \qquad \mathcal{R}_0 > 1.$$
 \end{proof}

\section{Stroboscopic maps and their fixed points}

In this section, we study the existence of fixed points associated with the stroboscopic maps (time $T$ maps) of the {\it Susceptible} and {\it Infectious}  individuals for system \eqref{modelo3}. We also analyze their stability in the sense of Subsection \ref{ss: definitions}.\\

For $n\in \NN$ and $T>0$, using \eqref{flow2D}, define the sequences:
$$
S_n= S\left( nT, (S_0, I_0) \right) \quad \text{and} \quad I_n= I \left( nT, (S_0, I_0) \right) ,
$$
which can be seen as the stroboscopic $T$-maps associated to the trajectory of $(S_0, I_0)\in (\RR_0^+)^2$.
 \bigbreak
\begin{lem}\label{ponto_fixo_I}  \, 
\begin{enumerate}
\item There exists $F_I:\RR_0^+ \rightarrow \RR^+_0$ such that $I_{n+1}=F_I(I_n)$ and 

$$F_I(y)= y \exp{\left\{\beta_0 \displaystyle \int_{nT}^{(n+1)T} S(t)\, \mathrm{d}t \right\}} e^{-\left( \sigma+g \right) T}.$$ 
\item The fixed point of $F_I$ is:

\smallskip

\begin{enumerate}
\item  $y^{\star}=0$ if \,$\dpt  \dfrac{1}{T} \int_0^T  \mathcal{S}(t) \, \mathrm{d}t \neq S_c$;
\medskip
\item (any) $y^{\star} \in \RR_0^+$ if \,$\dpt  \dfrac{1}{T} \int_0^T  \mathcal{S}(t) \, \mathrm{d}t = S_c$.
\end{enumerate}
\end{enumerate}
\end{lem}

\begin{proof}
\begin{enumerate}
\item   From \eqref{modelo3}, for $I(t)>0$ for all $t\in \RR_0^+$, we have
 
\begin{eqnarray*}
\nonumber && \dot{I}(t) = \beta_0 S(t) I(t)  - (\sigma+g) I(t) \\
\nonumber && \\
\nonumber \Leftrightarrow&&  \dfrac{\mathrm{d}I(t)}{I(t)}  = \left[ \beta_0 S(t) - (m+g) \right] \mathrm{d}t \\
\nonumber && \\
\nonumber \Rightarrow&& \displaystyle \int_{t_0}^t \dfrac{\mathrm{d}I(\tau)}{I(\tau)} = \beta_0 \displaystyle \int_{t_0}^t S(\tau) \mathrm{d}\tau  - (m+g) \displaystyle \int_{t_0}^t \, \mathrm{d}\tau  \\
 \nonumber && \\
\nonumber \Leftrightarrow&& \ln{I(t)} - \ln{I(t_0)} = \beta_0  \displaystyle \int_{t_0}^t S(\tau) \, \mathrm{d}\tau - (m+g)(t-t_0) \\
\nonumber && \\
\Leftrightarrow&& I(t) = I_0 \exp{\left\{\beta_0 \displaystyle \int_{t_0}^t S(\tau)\, \mathrm{d}\tau \right\}} e^{-\left(\sigma+g\right)\left(t-t_0\right)},
\label{S_tau} 
\end{eqnarray*}
where $I(t_0) \coloneqq I_0> 0$. The sequence $I_{n+1} = I \left( (n+1)T, (S_0, I_0)\right)$ is then computed as
\begin{eqnarray}
\nonumber &&  I_{n+1} \coloneqq I_n \exp{\left\{\beta_0 \displaystyle \int_{nT}^{(n+1)T} S(t)\, \mathrm{d}t \right\}} e^{-\left(\sigma+g\right)T} \\
\nonumber && \\
\nonumber \Leftrightarrow && I_{n+1}=F_I(I_n)  ,
 \end{eqnarray}
 
where $F_I(y)= y \exp{\left\{\beta_0 \displaystyle \int_{nT}^{(n+1)T} S(t)\, \mathrm{d}t \right\}} e^{-\left( \sigma+g \right) T}$. 

\medskip

 \item The fixed points of $F_I$ can be found by solving the equation $ F_I(y^{\star}) = y^{\star}$
which is equivalent to
\begin{eqnarray*}
\nonumber &&  y^{\star}\exp{\left\{\beta_0 \displaystyle \int_{nT}^{(n+1)T} S(t)\, \mathrm{d}t \right\}} e^{-\left(\sigma+g\right)T} = y^{\star} \\
&\Leftrightarrow & y^{\star}\exp \left( \dpt  \dfrac{1}{T} \int_0^T  S(t) \, \mathrm{d}t - S_c\right)= y^{\star}.
\end{eqnarray*}

\medskip

In other words, if \,$\dpt  \dfrac{1}{T} \int_0^T  S(t) \, \mathrm{d}t \neq S_c $, then the fixed point of $F_I$ is given by $y^{\star}=0$. Otherwise any $y^{\star} \in \RR_0^+$ is a solution of $F_I(y^{\star})=y^{\star}$.
\end{enumerate}
 \end{proof}
 \medbreak
 
  If \,$\dpt \dfrac{1}{T} \int_0^T  S(t) \, \mathrm{d}t \neq S_c $, then the unique fixed point of $F_I$ is $y^{\star}=0$.  From now on, we focus the analysis  on this fixed point.  We remind the definition of $p_1$ from the statement of Theorem \ref{thA}: $p_1\equiv p_1(T)=1-e^{-AT}$, $T> 0$.

\medbreak
\begin{lem} \label{estroboscopica_S}
If $I=0$, then:\\

\begin{enumerate}
\item There exists $F_S:\RR_0^+ \rightarrow \RR_0^+$ such that $S_{n+1}=F_S(S_n)$ and 

$$
F_S(x)=  \displaystyle \dfrac{A\,x \left(1-p \right)e^{AT}}{x \left(e^{AT} - 1\right)+A}.
$$

\medskip

\item The map $F_S(x)$ has two fixed points:

$$x_1^{\star} = A\left(1- \dfrac{pe^{AT}}{e^{AT} - 1} \right) \quad \text{and} \quad x_2^{\star} = 0.$$

\smallskip

For $T>0$, if $p > p_1(T)$, then $x_1^{\star}<0$.
\end{enumerate}
\end{lem}
\bigbreak

\begin{proof}
\begin{enumerate}
\item We study system \eqref{modelo3} assuming $I=0$, the unique fixed point of $F_I$ if $ \dpt  \dfrac{1}{T} \int_0^T  S(t) \, \mathrm{d}t \neq S_c $ (cf. Lemma \ref{ponto_fixo_I}).
The growth of $S$  for $t_0 = nT \leq t < (n+1)T$ is given by
\begin{equation}
\label{s_wt_i}
\begin{cases}
\medskip
&\dot{S}(t) = S(t)\left(A-S(t)\right), \,\,\,\,\,\,\,  \quad \, t \neq nT, \\
&S(nT) = (1-p) S(nT^-). \,\,\, \quad \,
\end{cases}
\end{equation}

\medskip

Integrating   \eqref{s_wt_i} between pulses, and assuming that $S(t)(A-S(t))\neq 0$, we obtain
\begin{eqnarray}
\nonumber && \dot{S}(t) = S(t)\left(A-S(t)\right) \\
\nonumber && \\
\nonumber \Leftrightarrow&& \displaystyle \int_{t_0}^t \dfrac{\mathrm{d}S(\tau)}{S(\tau)\left(A-S(\tau)\right)} = \displaystyle \int_{t_0}^t \,\mathrm{d}\tau    \\
\nonumber && \\
\Leftrightarrow&& \displaystyle \int_{t_0}^t \dfrac{\mathrm{d}S(\tau)}{S(\tau)\left(A-S(\tau)\right)} = t - t_0  \label{integrar_p_Susc}
\end{eqnarray}
and thus we get:
\begin{eqnarray}
S(t) = \displaystyle \dfrac{A\,S_0}{S_0 + \left(A-S_0\right)e^{-A\left(t-t_0\right)}}, \quad  \qquad t \geq t_0 \geq 0,
\label{S_expression}
\end{eqnarray}
 where $S(t_0) \coloneqq S_0$ and $A > S_0$. Notice that \eqref{S_expression} holds between pulses.
 Bearing in mind that
\begin{eqnarray}
S(nT^-) = \displaystyle \lim_{t \rightarrow nT^-} S(t), \,\,  \quad S(nT) = (1-p)S(nT^-),\,\,  \quad n \in \mathbb{N},
\label{notions}
\end{eqnarray}
 using induction over $n\in \NN$ it is easy to show that the general expression of $S(t)$ for    $nT \leq t < (n+1)T$ is

\begin{eqnarray}
\nonumber S(t) = \displaystyle \dfrac{A\left(1-p \right) S(nT^-)}{\left(1-p \right)S(nT^-) + \left[A-\left(1-p \right)S(nT^-) \right]e^{-A\left(t-nT\right)}}.
\label{s_t_geral_0}
\end{eqnarray}
\bigbreak

Using \eqref{notions} and considering $S_n= S\left( nT, (S_0, I_0) \right) $, we get
\begin{eqnarray}
S(t) =   \displaystyle \dfrac{A\,S_n}{S_n + \left( A - S_n \right) e^{-A\left(t-nT\right)}}, \quad 
\label{s_t_geral}
\end{eqnarray}

\medskip

\noindent where $ nT \leq t < (n+1)T$. In particular,  

\begin{eqnarray}
\nonumber  &&S((n+1)T^-) =  \displaystyle \dfrac{A\left(1-p \right) S(nT^-)}{\left(1-p \right)S(nT^-) + \left[A-\left(1-p \right)S(nT^-) \right]e^{-AT}} \\
\nonumber && \\
\nonumber && \\
\nonumber \overset{\eqref{notions}}{\Leftrightarrow} && \dfrac{S((n+1)T)}{1-p} =  \displaystyle \dfrac{A\left(1-p \right) S(nT^-)}{\left(1-p \right)S(nT^-) + \left[A-\left(1-p \right)S(nT^-) \right]e^{-AT}} \\
\nonumber && \\
\nonumber && \\
\nonumber \Leftrightarrow && S_{n+1} = \displaystyle \dfrac{A\,S_n \left(1-p \right)}{S_n + \left(A-S_n \right)e^{-AT}} \\
\nonumber && \\
\nonumber && \\
 &\Leftrightarrow& S_{n+1} = \displaystyle \dfrac{A\,S_n \left(1-p \right)}{S_n \left(1-e^{-AT} \right) + Ae^{-AT}} \label{Sn+1_0}.
\end{eqnarray}

 Multiplying both the numerator and the denominator of \eqref{Sn+1_0} by $e^{AT}>0$, we get
\begin{eqnarray}
\nonumber && S_{n+1} \coloneqq \displaystyle \dfrac{A\,S_n\left(1-p \right)e^{AT}}{S_n\left(e^{AT} - 1\right)+A} \qquad  \Leftrightarrow \qquad  S_{n+1}= F_S(S_n)  ,
  \label{Fx} 
\end{eqnarray}

where
$$
F_S(x)=  \displaystyle \dfrac{A\,x \left(1-p \right)e^{AT}}{x \left(e^{AT} - 1\right)+A}.
$$

\bigbreak

\item  The map $F_S$ has two fixed points:
\begin{eqnarray}
\label{fixed_points_}
\nonumber  x_1^{\star} &=& A\left(1- \dfrac{pe^{AT}}{e^{AT} - 1} \right) \\
\nonumber  x_2^{\star} &=& 0 ,
\end{eqnarray}

\smallskip

where $x_1^{\star}>0$ if and only if

$$
A\left(1- \dfrac{pe^{AT}}{e^{AT} - 1} \right) > 0 
\quad \Leftrightarrow \quad p < 1 - e^{-AT}  \quad \Leftrightarrow \quad p < p_1(T).
$$
\end{enumerate}
 \end{proof} 
 
Combining Lemmas \ref{ponto_fixo_I} and \ref{estroboscopica_S}, we know that if \,$\dpt  \dfrac{1}{T} \int_0^T  S(t) \, \mathrm{d}t \neq S_c $, then  $(F_S, F_I)$ has two non-negative fixed points: $(0,0)$ and $(x_1^\star, 0)$. In the flow of  \eqref{modelo3}, they  are  denoted by $(0,0)$ (stationary trivial equilibrium) and $(\mathcal{S}, 0)$ (periodic non-trivial disease-free solution).
\begin{lem}
\label{Expression1}
For $T>0$, the non-trivial $T$-periodic solution of \eqref{modelo3} associated to $(x_1^\star, 0)$ does not depend on the transmission rate $\beta$ and is  parametrised by
$$
(\mathcal{S}(t),0)=  \left(\dfrac{A\left[e^{AT}\left(1-p\right)-1\right]}{e^{AT}\left(1-p\right)-1+pe^{A\left(T-(t-t_0)\right)}}, 0\right),
$$
where $t_0 = nT \leq t < (n+1)T$ and $n\in \NN$. 
\end{lem}

\begin{proof}

Replacing $S_0$ by  $x_1^{\star}$ in \eqref{S_expression}, we obtain
\begin{eqnarray}
\nonumber \mathcal{S}(t) &=& \displaystyle \dfrac{A\, x_1^{\star}}{x_1^{\star} + \left(A-x_1^{\star}\right)e^{-A (t-t_0)}} =  \dfrac{A\left[e^{AT}\left(1-p\right)-1\right]}{e^{AT}\left(1-p\right)-1+pe^{A\left(T-(t-t_0)\right)}},\label{S_infect_free}
\end{eqnarray}
 where $t_0 = nT \leq t < (n+1)T$ and $n\in \NN$.  By construction, this solution is {\it $T$-periodic} (a consequence of Lemmas   \ref{ponto_fixo_I} and \ref{estroboscopica_S}).
\end{proof}

\section{Basic reproduction number    $\mathcal{R}_p$ for \eqref{modelo3}}
\label{section:Rp}

For $T>0$, following \cite[Eq. (3.15)]{LuChiChen2002}, we  compute  the {\it basic reproduction number} $\mathcal{R}_p$ in the absence of seasonality for   \eqref{modelo3} as
\begin{eqnarray}
\mathcal{R}_p(T):=\dpt \dfrac{\beta_0}{\sigma + g}\dfrac{1}{T} \int_0^T  \mathcal{S}(t) \, \mathrm{d}t =  \dfrac{1}{S_c T} \int_0^T  \mathcal{S}(t) \, \mathrm{d}t ,
\label{Rp_eq} 
\end{eqnarray}

\noindent based on the  disease-free periodic solution $\mathcal{S}(t)$ given explicitly in Lemma \ref{Expression1}.
The quantity $\mathcal{R}_p$ (when less than 1) can be used to measure the velocity at which the disease is eradicated. For clarity, we omit the dependence of $\mathcal{R}_p$ on $T>0$.
In the sequel, we present a series of properties of $\mathcal{R}_p$.

\begin{lem}
\label{Lema7}
With respect to  \eqref{modelo3},  the following assertions hold for $p\in [0,1)$ and $T>0$: \\
\begin{enumerate}
\item $\mathcal{R}_p =    \dpt \mathcal{R}_0 \left[ \dfrac{\ln{(1-p)}}{AT} + 1 \right]$; 
\medskip
\item if $\mathcal{R}_p=1$, then $F_I$ has infinitely many fixed points;
\medskip
\item for $T>0$ and $p>0$, we have $\frac{\partial \mathcal{R}_p}{\partial T}(T)>0$ and \,$\frac{\partial \mathcal{R}_p}{\partial p}(T)<0$;
\medskip
\item $\displaystyle \lim_{T \rightarrow +\infty} \mathcal{R}_p = \mathcal{R}_0$(\footnote{This limit on $T$ may be interpreted as the absence of vaccination.});
\medskip
\item If $p=0$, then $\mathcal{R}_p=\mathcal{R}_0$.\\
 
\end{enumerate}
\end{lem}

\begin{proof}

\begin{enumerate}
\item 
For the sake of simplicity, let us assume $t_0 = 0$. Then we have
\begin{eqnarray*}
\int_0^T \mathcal{S}(t) \, \mathrm{d}t \overset{\eqref{S_infect_free}}{=} \int_0^T \dfrac{A\left[e^{AT}\left(1-p\right)-1\right]}{e^{AT}\left(1-p\right)-1+pe^{A\left(T-t\right)}} \, \mathrm{d}t =  \ln{\left( 1-p \right)} + AT.
\label{expre_integral_S_til}
\end{eqnarray*}

Therefore,

\begin{eqnarray}
\mathcal{R}_p &=&  \dpt \dfrac{\beta_0}{\sigma + g}\dfrac{1}{T} \int_0^T  \mathcal{S}(t) \, \mathrm{d}t \label{RpAND_Sc} \\
\nonumber && \\
\nonumber && \\
\nonumber& {=}&  \dfrac{\beta_0}{\sigma + g}\dfrac{1}{T} \big[ \ln{(1-p)} + AT \big] \label{Rp1} \\
\nonumber && \\
\nonumber && \\
\nonumber &\overset{\eqref{R0Sc}}{=}& \dfrac{1}{S_c T} \big[ \ln{(1-p)} + AT \big] \label{Rp1} \\ 
\nonumber && \\
\nonumber && \\
 \nonumber&\overset{\eqref{R0},\eqref{R0Sc}}{=}&  \dpt \mathcal{R}_0 \left[ \dfrac{\ln{(1-p)}}{AT} + 1 \right] \label{Rp2}.
\end{eqnarray}

\item Since
\begin{eqnarray}
\nonumber \mathcal{R}_p=1 &\Leftrightarrow &   \dpt  \dfrac{1}{T} \int_0^T  \mathcal{S}(t) \, \mathrm{d}t=S_c\\
\nonumber \\
\nonumber   &\Leftrightarrow &   \exp{\left\{\beta_0 \displaystyle \int_{nT}^{(n+1)T} \mathcal{S}(t)\, \mathrm{d}t \right\}} e^{-\left( \sigma+g \right) T} = 1,
 \end{eqnarray}
the result follows by observing the analytic expression of $F_I$ in Lemma \ref{ponto_fixo_I}.
\bigbreak

\item From the proof of item {\it (1)}, one knows that $\mathcal{R}_p =  \mathcal{R}_0 \left[ \frac{\ln{(1-p)}}{AT} + 1 \right] $. If $p>0$, then

$$
\nonumber \frac{\partial \mathcal{R}_p}{\partial T}(T) = -\dfrac{\mathcal{R}_0\ln{\left(1-p\right)}}{AT^2} >  0 \qquad \text{and} \qquad \frac{\partial \mathcal{R}_p}{\partial p}(T) = -\dfrac{\mathcal{R}_0}{AT \left(1-p\right)} < 0 .
$$

\end{enumerate}

Items {\it (4)} and {\it (5)} are straightforward using the formula of $\mathcal{R}_p$ of item  {\it (1)}.\\
\end{proof}
 
The following useful result relates the stability of the impulsive periodic solution of Lemma \ref{Expression1} with the {\it basic reproduction number} $\mathcal{R}_p$ for \eqref{modelo3}.  
 
\begin{lem} \label{e_qui_va_lence}
With respect to system \eqref{modelo3}, if \,$A > S_c$ (\footnote{$A>S_c \,\, \Leftrightarrow \,\, \mathcal{R}_0 > 1$, by Lemma \ref{A>S_c}.}), then \,$\mathcal{R}_p < 1 \,\,  \Leftrightarrow \,\, p > p_2.$
\end{lem}
 
\begin{proof}
The proof follows from:
$$
\mathcal{R}_p < 1  \Leftrightarrow   \dfrac{1}{S_c T} \big[ \ln{(1-p)} + AT \big] < 1  \Leftrightarrow   p > p_2.
$$
\end{proof}
   
\section{Local stability of the disease-free  solutions} \label{Section_of_stability}
 The next result proves the local asymptotic stability of the disease-free periodic solutions of \eqref{modelo3}, $(0,0)$ and $(\mathcal{S}, 0)$,  using {\it Floquet multipliers}.

\begin{prop} \label{stable_periodic_solution}
 With respect to  \eqref{modelo3}, the following assertions hold for $T>0$:
 
\begin{enumerate}
\item the disease-free trivial solution $(0,0)$ is locally asymptotically stable provided  $p > p_1(T)$;
\smallskip
\item the disease-free periodic solution $(\mathcal{S},0)$ is locally asymptotically stable provided  $p_2(T) < p < p_1(T)$.
\end{enumerate}
\end{prop}

\begin{proof}
The stability of the disease-free periodic solution is found by  analyzing the behavior of initial conditions close to it. This is achievable by computing the  {\it monodromy matrix} (see \cite[Chapter II, pp. 28]{Bainov_livro}). If the absolute value of the eigenvalues ({\it Floquet multipliers}) of the monodromy matrix is less than one, then the periodic solution $(\mathcal{S},0)$ is asymptotically stable \cite[Theorem 3.5, pp. 30]{Bainov_livro}. 

 We  exhibit  the computations near the {\it $T$-periodic} solution  $(\mathcal{S},0)$.
For $t\geq 0$, we set
 \begin{eqnarray}
\label{linear_7}
\nonumber S(t) &=& \mathcal{S}(t) + s(t) \\
\nonumber I(t) &=& \mathcal{I}(t) + i(t) \quad \Leftrightarrow \quad I(t) \,\,\,\, = \,\,\,\, i(t) ,
\end{eqnarray}
where $s(t)$ and $i(t)$ are small terms close to 0. We will show that they vanish when $t$ increases.
Omitting the dependence of the variables on $t$, from \eqref{modelo3} and \eqref{linear_7}, one gets:
\begin{eqnarray*}
\nonumber \dfrac{\mathrm{d}s}{\mathrm{d}t} &=& \dfrac{\mathrm{d}S}{\mathrm{d}t} - \dfrac{\mathrm{d}\mathcal{S}}{\mathrm{d}t} \\
\nonumber && \\
\nonumber &=& S(A-S) - \beta_0 S I - \big[ \mathcal{S}(A - \mathcal{S}) - \beta_0 \mathcal{S} \mathcal{I} \big]  \\
\nonumber && \\
\nonumber &\overset{\mathcal{I} =0}{=}& SA - S^2 - \beta_0 S I - \mathcal{S}A + \mathcal{S}^2 \\
\nonumber && \\
\nonumber &=& - \mathcal{S}A + \mathcal{S}^2 + A(\mathcal{S} + s) - (\mathcal{S} + s)^2 - \beta_0(\mathcal{S} + s) i  \\
\nonumber && \\
\nonumber &=& - \mathcal{S}A + \mathcal{S}^2 + \mathcal{S}A + sA - \mathcal{S}^2 - 2s\mathcal{S} - s^2 - \beta_0 \mathcal{S} i - \beta_0 s i    \\
\nonumber && \\
\nonumber &=& As - 2 s \mathcal{S} - \beta_0 \mathcal{S} i - \beta_0 s i - s^2\\
\nonumber && \\
&=& (A - 2\mathcal{S} - \beta_0 i) s - \beta_0 \mathcal{S} i - s^2\label{LINEAR_S_I} \\
\nonumber && \\
\nonumber \text{and} \\
\nonumber && \\
\nonumber \dfrac{\mathrm{d}i}{\mathrm{d}t} &=& \dfrac{\mathrm{d}I}{\mathrm{d}t} - \dfrac{\mathrm{d}\mathcal{I}}{\mathrm{d}t} \\
\nonumber && \\
\nonumber &=& \beta_0 \mathcal{S} i - (\sigma + g) i  \\
\nonumber && \\
&=& \big[ \beta_0 \mathcal{S} - (\sigma + g) \big] i . \label{LINEAR_I}
\end{eqnarray*}

Hence, we model the evolution of $(s, i)$ as
\begin{equation}
\label{linear_system_}
\begin{array}{lcl}
\begin{cases}
\bigskip
&\dfrac{\mathrm{d}s}{\mathrm{d}t} = (A - 2\mathcal{S} - \beta_0 i) s - \beta_0 \mathcal{S} i  \\
\bigskip
\bigskip
&\dfrac{\mathrm{d}i}{\mathrm{d}t} = \big[ \beta_0 \mathcal{S} - (\sigma + g) \big] i \, , \qquad \qquad \qquad \qquad t \neq nT,  \\
\medskip
&s(nT) = (1-p) s(nT^-)  \\
\medskip
&i(nT) = i(nT^-).
\end{cases}
\end{array}
\end{equation}

Note that $\mathcal{S}(t)$ is known explicitly (Lemma \ref{Expression1}) and may be written as the periodic coefficient of $s(t)$ and $i(t)$. 
Lyapunov's theory neglects quadratic terms   to compute the stability of hyperbolic fixed points \cite[\S I6, p. 6]{IoossJoseph1980}.  
The solutions of \eqref{linear_system_} may be written as 
\begin{equation*}
\label{m1}
\begin{array}{lcl}
\left(\begin{array}{c}
s(t) \\ 
\\
i(t)
\end{array}\right)
= 
\Phi(t) \left(\begin{array}{c}
s(0) \\ 
\\
i(0)
\end{array}\right) 
\end{array}
\quad \text{where} \quad
\begin{array}{lcl}
\Phi(t) = \left(\begin{array}{cc}
\varphi_{1,1}(t) & \varphi_{1,2}(t) \\ 
\\
\varphi_{2,1}(t) & \varphi_{2,2}(t)
\end{array}\right) 
\end{array}
\end{equation*}
is the fundamental matrix whose columns are the components of linearly independent solutions of \eqref{linear_system_}, and $\Phi(0)  $  is the identity matrix. According to Floquet theory \cite[\S VII.6.2, p. 146]{IoossJoseph1980}, $\Phi(t)$ satisfies
\begin{equation*}
\label{m2}
\begin{array}{lcl}
\dfrac{\mathrm{d}\Phi}{\mathrm{d}t} \Big|_{(s^\star, i^\star)}&=& \mathcal{J}(t) \Phi(t) \\
\\
&=& \left(\begin{array}{cc}
A - 2 \mathcal{S}(t) - \beta_0 i^{\star} & -\beta_0 s^{\star} -\beta_0 \mathcal{S}(t) \\ 
\\
\\
0 & \beta_0 \mathcal{S}(t) - \left(\sigma+g\right)
\end{array}\right) \Phi(t) \\
\\
\\
&\overset{(s^{\star},i^{\star})\,=\,(0,0)}{=}& \left(\begin{array}{cc}
A - 2 \mathcal{S}(t) & -\beta_0 \mathcal{S}(t) \\ 
\\
\\
0 & \beta_0 \mathcal{S}(t) - \left(\sigma+g\right)
\end{array}\right) \Phi(t) ,
\end{array}
\end{equation*}
where $\mathcal{J}(t)$ is the Jacobian matrix of \eqref{linear_system_} around   $(s^{\star},i^{\star})=(0,0)$. The fundamental matrix $\Phi(t)$ can be written as

\begin{eqnarray*}
\nonumber \Phi(t)  &=& \left(\begin{array}{cc}
\exp{ \left\{ AT - 2 \dpt \int_0^T \mathcal{S}(t) \, \mathrm{d}t \right\} } & \varphi_{1,2}(t) \\ 
\\
\\
0 & \exp{ \left\{ \beta_0 \dpt \int_0^T \mathcal{S}(t) \, \mathrm{d}t - \left(\sigma + g \right)T \right\} }
\end{array}\right).
\end{eqnarray*}
The exact form of $\varphi_{1,2}(t)$ is not necessary since it is not needed in the subsequent analysis. Since the linearisation of the third and fourth equations of \eqref{linear_system_} results in

\begin{equation*}
\label{m3}
\begin{array}{lcl}
\left(\begin{array}{c}
s(nT) \\ 
\\
i(nT)
\end{array}\right)
=
\left(\begin{array}{cc}
1-p &0 \\ 
\\
0 & 1
\end{array}\right)
\left(\begin{array}{c}
s(nT^-) \\ 
\\
i(nT^-)
\end{array}\right) ,
\end{array}
\end{equation*}

\noindent then the {\it Floquet multipliers} of $(\mathcal{S},0)$ are the  solutions in $\lambda$ of:

\begin{eqnarray}
\small
   \det{\left[\begin{array}{cc}
 (1-p) \, \exp{ \left\{ AT - 2 \dpt \int_0^T \mathcal{S}(t) \, \mathrm{d}t  \right\} } - \lambda & \varphi_{1,2}(t) \\ 
\\
\\
\nonumber 0 & \exp{ \left\{ \beta_0 \dpt \int_0^T \mathcal{S}(t) \, \mathrm{d}t - \left(\sigma + g \right)T \right\} } - \lambda
\end{array}\right]} &=& 0. 
\end{eqnarray}

\noindent They are explicitly given by

\begin{eqnarray}
\lambda_1 &=& (1-p) \, \exp{ \left\{ AT - 2 \dpt \int_0^T \mathcal{S}(t) \, \mathrm{d}t \right\}}>0, \label{eigenvalue1}  \\
\nonumber && \\
\lambda_2 &=& \exp{ \left\{ \beta_0 \dpt \int_0^T \mathcal{S}(t) \, \mathrm{d}t - \left(\sigma + g \right)T \right\}}>0.\label{eigenvalue2}
\end{eqnarray}

\bigbreak

\textbf{\textit{Floquet multipliers}} \textbf{associated to \boldsymbol{$(\mathcal{S},0)$}:}
From \eqref{eigenvalue1} we get
\begin{eqnarray*}
\nonumber  && \lambda_1<1\\
\nonumber && \\
\nonumber &\Leftrightarrow & (1-p) \, \exp{ \left\{ AT - 2 \dpt \int_0^T \mathcal{S}(t) \, \mathrm{d}t \right\} } < 1 \\
\nonumber && \\
\nonumber &\Leftrightarrow& \ln{(1-p)} + AT > 0 \\ 
\nonumber && \\
&\Leftrightarrow& p < p_1 \label{L1122},
\end{eqnarray*}

and from \eqref{eigenvalue2} we deduce that
\begin{eqnarray}
\nonumber  && \lambda_2<1\\
\nonumber && \\
\nonumber &\Leftrightarrow & \exp{ \left\{ \beta_0 \dpt \int_0^T \mathcal{S}(t) \, \mathrm{d}t - \left(\sigma + g \right)T \right\} } < 1 \\
\nonumber && \\
\nonumber&\Leftrightarrow&           \dfrac{1}{T} \int_0^T  \mathcal{S}(t) \, \mathrm{d}t <  \dfrac{\sigma + g}{\beta_0}  \overset{\eqref{R0Sc}}{=} S_c    \\ 
\nonumber && \\
\nonumber  &\Leftrightarrow& \dfrac{\beta_0}{\sigma + g}\frac{1}{T} \int_0^T  \mathcal{S}(t) \, \mathrm{d}t < 1 \\
\nonumber && \\
\nonumber &\overset{\eqref{RpAND_Sc}}{\Leftrightarrow}& \mathcal{R}_p <1 .
\end{eqnarray}

\medskip

\textbf{\textit{Floquet multipliers}} \textbf{associated to \boldsymbol{$(0,0)$}:}

$$
  \lambda_1 = (1-p) \, e^{AT}   \quad \text{and} \quad \lambda_2=e^{- \left(\sigma + g \right)T} .
$$
 
\medbreak It is easy to check that $\lambda_1<1$ if and only if $p>p_1(T)$ and $\lambda_2<1$. The result is    shown.
\end{proof}

 The maps $p_1$ and $p_2$ given in Theorem \ref{thA} admit inverse. Let us denote them by $T_1$ and $T_2$, respectively, which  can be explicitly given by

\begin{eqnarray*}
\label{T_max_}
T_1 (p)\coloneqq \dfrac{\left| \ln{(1-p)} \right|}{A} \quad \text{and} \quad T_{\text{2}}(p) \coloneqq \dfrac{\left| \ln{(1-p)} \right|}{A - S_c} \,\, , \qquad \text{for} \,\,\, A > S_c .
\end{eqnarray*}

\medskip

  Omitting the dependence of $T_1$ and $T_2$ on $p$, the following result is a consequence of Proposition \ref{stable_periodic_solution}. It characterizes the stability of $(0,0)$ and $(\mathcal{S},0)$ as function of $T$ (cf. Figure \ref{Bif_Diag_ThA}).
  \bigbreak
 
\begin{cor} \label{Tmax_Lemma}
With respect to  \eqref{modelo3}, the following assertions hold for $p >0$:

\medbreak

\begin{enumerate}
\item If \,$T<T_1(p)$, then there are no non-trivial periodic solutions; the trivial equilibrium  $(0,0)$ is stable;
\smallskip
\item If \,$T_1(p) < T < T_{2}(p)$, then the disease-free periodic solution $(\mathcal{S},0)$ is stable and the trivial equilibrium  $(0,0)$ is unstable;
\smallskip
\item If \,$T>T_2(p)$, then both the periodic solution  $(\mathcal{S},0)$ and the trivial equilibrium  $(0,0)$ are unstable.
\end{enumerate}
\end{cor}

 \section{Proof of (2) of Theorem \ref{thA}}
 \label{section: item (2)}
For $T>0$, if $p>p_1$, the unique compact and invariant set for the stroboscopic map $(F_S, F_I)$ is $(0,0)$. This is the unique candidate for the {\it $\omega$-limit} of a trajectory of a planar differential equation (cf. {\cite{Bonotto2008}})

\mathversion{bold}
\section{Proof of (3) of Theorem \ref{thA}}
\label{Section_of_global_stability}
\mathversion{normal}

We prove the global stability of the disease-free periodic solution $(\mathcal{S}, 0)$ given in Lemma \ref{Expression1}.  
By Proposition \ref{stable_periodic_solution}, we know  that  if $p\in (p_2(T), p_1(T))$, then the  non-trivial periodic solution $(\mathcal{S},0)$ is locally asymptotically stable.
Suppose that $(S(t), I(t))$, $t\in \RR_0^+$ is a solution of  \eqref{modelo3} with positive initial conditions in $\mathcal{M}$ (given in Lemma \ref{region_}).

\begin{lem}
\label{lema10}
 If \,$\mathcal{R}_p < 1$, then  $\displaystyle \lim_{t \rightarrow +\infty} I(t) = 0$. 
\end{lem}

\begin{proof}
According to Lemma \ref{estroboscopica_S} we know that $x_1^{\star}>0 $  if and only if  \,$p < p_1$. From the first and third equations of \eqref{modelo3}, we see that, for any $0 < \varepsilon\ll 1$  there exists $T^{\dagger} \gg 1$, such that
\begin{eqnarray}
S(t) < \mathcal{S}(t) + \varepsilon ,  \label{jin3.1}
\end{eqnarray}
for all $t > T^{\dagger}$.    Substituting \eqref{jin3.1} into the second equation of \eqref{modelo3}, we obtain
\begin{eqnarray}
\nonumber \dfrac{\mathrm{d}I(t)}{\mathrm{d}t} &\leq& \beta_0 \left( \mathcal{S}(t) + \varepsilon \right) I(t) - \left(\sigma + g \right)I(t) .
\end{eqnarray}
Since $I(nT) = I(nT^-)$, for $t \in [T^{\dagger} + nT, T^{\dagger} + (n+1)T)$, we have
\begin{eqnarray}
\nonumber && \dfrac{\mathrm{d}I(t)}{\mathrm{d}t} \leq \left[ \beta_0 \left( \mathcal{S}(t) + \varepsilon \right) - \left(\sigma + g \right) \right] I(t) \\
\nonumber \\
\nonumber &\Rightarrow& \int_{T^{\dagger}}^t \dfrac{\mathrm{d}I(\tau)}{I(\tau)} \leq \int_{T^{\dagger}}^t \left[ \beta_0 \left( \mathcal{S}(\tau) + \varepsilon \right) - \left(\sigma + g \right) \right] \mathrm{d}\tau \\
\nonumber \\
\nonumber &\Leftrightarrow& \ln{I(t)} - \ln{I(T^{\dagger})} \leq \beta_0 \int_{T^{\dagger}}^t \left( \mathcal{S}(\tau) + \varepsilon \right) \mathrm{d}\tau - \left(\sigma + g \right)  \left(t - T^{\dagger} \right) \\
\nonumber \\
 &\Leftrightarrow& I(t) \leq I(T^{\dagger}) \exp{\left\{ \beta_0 \int_{T^{\dagger}}^t \left( \mathcal{S}(\tau) + \varepsilon \right)  \mathrm{d}\tau - \left(\sigma + g \right)  \left(t - T^{\dagger} \right)  \right\}}. \label{jin3.2}
\end{eqnarray}

By hypothesis, one knows that  $ \mathcal{R}_p < 1$. Since
\begin{eqnarray}
\nonumber &&\mathcal{R}_p < 1\\
\nonumber \\
\nonumber &\overset{\varepsilon>0\, \, \text{suf. small}}{\Leftrightarrow}& \dfrac{\beta_0}{\left( \sigma + g \right)} \dfrac{1}{\left( t-T^{\dagger} \right)} \int_{T^{\dagger}}^t \left( \mathcal{S}(\tau) + \varepsilon \right)  \mathrm{d}\tau < 1 \\
\nonumber \\
\nonumber & \Leftrightarrow&\beta_0 \int_{T^{\dagger}}^t \left( \mathcal{S}(\tau) + \varepsilon \right)  \mathrm{d}\tau - \left(\sigma + g \right)  \left(t - T^{\dagger} \right) < 0,
\end{eqnarray}
  and $ I(T^{\dagger})\geq 0$, then from \eqref{jin3.2} we conclude that $\displaystyle \lim_{t \rightarrow \infty} I(t) = 0$.
\end{proof}
\subsection*{Auxiliary variable}
Since $S(t)>0$ and $\mathcal{S}(t)>0$, for all $t\in \RR_0^+$, we set:
\begin{eqnarray}
\label{Change11}
x(t) = \ln\left({\dfrac{S(t)}{\mathcal{S}(t)}}\right) \quad \Leftrightarrow \quad S(t) = e^{x(t)} \mathcal{S}(t) .  \label{mud_var_ln}
\end{eqnarray}

\begin{lem} \label{mud_var_ln_3} 
 $ x'(t)=- \mathcal{S}(t) \left( e^{x(t)} - 1 \right) - \beta_0 I(t)$
\end{lem}

\begin{proof}
The proof  follows from folklore derivative  computations. Indeed,\\
\begin{eqnarray}
\nonumber  x'(t) & \overset{\mathcal{I}(t) = 0}{=}& \dfrac{S(t) \left( A - S(t) \right) - \beta_0 S(t) I(t)}{S(t)} - \dfrac{\mathcal{S}(t)\left(A - \mathcal{S}(t) \right)}{\mathcal{S}(t)} \\
\nonumber \\
\nonumber  &=& A - S(t) - \beta_0 I(t) - A + \mathcal{S}(t) \\
\nonumber \\
\nonumber &\overset{\eqref{mud_var_ln}}{=}& -e^{x(t)} \mathcal{S}(t) - \beta_0 I(t) + \mathcal{S}(t) \\
\nonumber \\
\nonumber &=&- \mathcal{S}(t) \left( e^{x(t)} - 1 \right) - \beta_0 I(t) .
\end{eqnarray}
\end{proof}

\noindent Lemma \ref{lema10} states that,  if $\mathcal{R}_p < 1$, then all solutions approach a disease free state ($I(t)=0$). Evaluating the equality of Lemma \ref{mud_var_ln_3} when $I=0$, we get: $ x'(t)\leq- \mathcal{S}(t) \left( e^{x(t)} - 1 \right).$

\begin{lem}  For $T^{\dagger} \gg 1$ of the proof of Lemma \ref{lema10}, we have:\\
\label{lema12}
 \,
\begin{enumerate}
\item 
$
 \dpt  \int_{T^{\dagger}}^t \dfrac{\mathrm{d}x(\tau)}{e^{x(\tau)} - 1} \leq - \int_{T^{\dagger}}^t  \mathcal{S}(\tau)\mathrm{d}\tau \\
$
\item $ \dpt \int_{T^{\dagger}}^t \dfrac{\mathrm{d}x(\tau)}{e^{x(\tau)} - 1}= - \left[ \left( x(t) - x(T^{\dagger}) \right) - \left( \ln{| e^{x(t)} - 1 |}   -  \ln{| e^{x(T^{\dagger})} - 1|} \right) \right]$\\
\end{enumerate}
\end{lem}

\begin{proof}
\begin{enumerate}
\item From Lemma \ref{mud_var_ln_3}, since $\beta_0 I(t) \geq 0$, we may write

\begin{eqnarray}
\nonumber x'(t) \leq - \mathcal{S}(t) \left( e^{x(t)} - 1 \right) , \label{mud_var_ln_4}
\end{eqnarray}

for $t \in [T^{\dagger} + nT, T^{\dagger} + (n+1)T)$, which implies: 

\begin{eqnarray}
\nonumber && \dfrac{\mathrm{d}x(t)}{\mathrm{d}t} \leq - \mathcal{S}(t) \left( e^{x(t)} - 1 \right) \\
\nonumber \\
\nonumber& & \int_{T^{\dagger}}^t \dfrac{\mathrm{d}x(\tau)}{e^{x(\tau)} - 1} \leq - \int_{T^{\dagger}}^t  \mathcal{S}(\tau)\mathrm{d}\tau .
\end{eqnarray}
\item \begin{eqnarray}
\nonumber \int_{T^{\dagger}}^t \dfrac{\mathrm{d}x(\tau)}{e^{x(\tau)} - 1} &=&   - \int_{T^{\dagger}}^t \dfrac{e^{x(\tau)} -1 - e^{x(\tau)}}{e^{x(\tau)} - 1} \, \mathrm{d}x(\tau) \\
\nonumber \\
\nonumber &=& - \left[  \int_{T^{\dagger}}^t \mathrm{d}x(\tau) -  \int_{T^{\dagger}}^t \dfrac{e^{x(\tau)}}{e^{x(\tau)} - 1} \mathrm{d}x(\tau) \right] \\
\nonumber \\
\nonumber &=& - \left[ \left( x(t) - x(T^{\dagger}) \right) - \left( \ln{| e^{x(t)} - 1 |}   -  \ln{| e^{x(T^{\dagger})} - 1|} \right) \right] .
\end{eqnarray}
\end{enumerate}
\end{proof}
For  $T^{\dagger} \gg 1$ of the proof of Lemma \ref{lema10}, define the map $h: [T^{\dagger}, +\infty) \rightarrow \RR$  as
 \begin{eqnarray}
\nonumber h(t) = \left( e^{x(T^{\dagger})} - 1 \right) \exp{\left\{-\left(    x(T^{\dagger}) + \displaystyle\int_{T^{\dagger}}^t  \mathcal{S}(\tau)\mathrm{d}\tau  \right)\right\}} .
\end{eqnarray}
\bigbreak
\begin{lem} 
\label{lema14A}\, \medbreak
\begin{enumerate}
\item $\dpt \lim_{t\rightarrow +\infty} h(t)=0$\\
\smallskip
\item For all $t>T^{\dagger}$, we get $ x(t) \leq - \ln{\left( 1-h(t) \right)}. $\\
\end{enumerate}

\end{lem}
\bigbreak
\begin{proof}
\begin{enumerate}
\item Since
\begin{eqnarray}
\label{lim_ht_1}
\nonumber && \dpt \lim_{t\rightarrow +\infty} \displaystyle \int_{T^{\dagger}}^t  \mathcal{S}(\tau)\mathrm{d}\tau \\
\nonumber \\
\nonumber  &=& {\small \dpt \lim_{t\rightarrow +\infty}\Big( -\ln{\big(pe^{A(T-T^{\dagger})} - pe^{AT} + e^{AT}-1 \big)} +  \ln{\big(pe^{A(T-t)} + e^{AT}(1-p) -1\big)} }\\
\nonumber &&+ A(t - T^{\dagger}) \Big) = +\infty,
\end{eqnarray}

then   $\dpt \lim_{t\rightarrow +\infty} h(t)=0$.  

\smallskip

\item Using Lemma \ref{lema12}, one gets
\begin{eqnarray} 
\nonumber && \int_{T^{\dagger}}^t \dfrac{\mathrm{d}x(\tau)}{e^{x(\tau)} - 1} \leq - \int_{T^{\dagger}}^t  \mathcal{S}(\tau)\mathrm{d}\tau \\
\nonumber \\
\nonumber &\overset{\text{Lemma \ref{lema12}}}{\Leftrightarrow}& - \left[ \left( x(t) - x(T^{\dagger}) \right) - \left( \ln{| e^{x(t)} - 1 |}   -  \ln{| e^{x(T^{\dagger})} - 1|} \right) \right] \leq - \int_{T^{\dagger}}^t  \mathcal{S}(\tau)\mathrm{d}\tau \\
\nonumber \\
\nonumber &\Leftrightarrow& \ln{| e^{x(t)} - 1 |} \leq \ln{| e^{x(T^{\dagger})}-1|} + \left(x(t) - x(T^{\dagger})  \right) - \int_{T^{\dagger}}^t  \mathcal{S}(\tau)\mathrm{d}\tau \\
\nonumber \\
\nonumber &\Leftrightarrow& -x(t) \geq \ln{\left( 1-\left( e^{x(T^{\dagger})} - 1 \right) \exp{\left\{-\left(    x(T^{\dagger}) + \int_{T^{\dagger}}^t  \mathcal{S}(\tau)\mathrm{d}\tau  \right)\right\}}\right)} \\
\nonumber \\
\nonumber &\Leftrightarrow& x(t) \leq - \ln{\left( 1-h(t) \right)} .
\end{eqnarray}
\end{enumerate}
 \end{proof}
 On the other hand, since $\displaystyle \lim_{t \rightarrow \infty} I(t) =0$, for any $0 < \varepsilon \ll 1$, there exist $T^{\dagger} > 0$, such that $I(t) \leq  \varepsilon$ for  $t>T^{\dagger}$. From Lemma \ref{mud_var_ln_3} we conclude that 
\begin{eqnarray}
\nonumber  x'(t) \geq - \mathcal{S}(t) \left( e^{x(t)} - 1 + \varepsilon \right)  , \label{mud_var_ln_6} 
\end{eqnarray}

\noindent  and hence

\begin{eqnarray}
\int_{T^{\dagger}}^t \dfrac{\mathrm{d}x(\tau)}{e^{x(\tau)} - 1 + \varepsilon} \geq - \int_{T^{\dagger}}^t  \mathcal{S}(\tau)\mathrm{d}\tau .\label{mud_var_ln_7}
\end{eqnarray}

In the next lemma, we evaluate the integral of the left hand side of \eqref{mud_var_ln_7}.

\begin{lem}
 $$  \int_{T^{\dagger}}^t \dfrac{\mathrm{d}x(\tau)}{e^{x(\tau)} - 1 + \varepsilon} = - \left[ \left( x(t) - x(T^{\dagger}) \right) - \left( \ln{| e^{x(t)} - 1 + \varepsilon |}   -  \ln{| e^{x(T^{\dagger})} - 1 + \varepsilon|} \right) \right] $$
\end{lem}

\bigbreak
\noindent The proof follows the same lines as item (2) of Lemma \ref{lema12}. For $\varepsilon \gtrsim 0$ and $T^{\dagger} \gg 1$ found in the proof of Lemma \ref{lema10}, define the map $h_\varepsilon: [T^{\dagger}, +\infty) \rightarrow \RR$ as
\begin{eqnarray}
\nonumber h_{\varepsilon}(t) = \left( e^{x(T^{\dagger})} - 1 + \varepsilon \right) \exp{\left\{-\left(    x(T^{\dagger}) + \int_{T^{\dagger}}^t  \mathcal{S}(\tau)\mathrm{d}\tau  \right)\right\}}.
\end{eqnarray}
 
\bigbreak
\begin{lem}\, \medbreak
\label{lema16A}
 
\begin{enumerate}
\item $\dpt \lim_{t\rightarrow +\infty} h_{\varepsilon}(t) =0$

\medskip

\item The following inequalities are equivalent: 

\medskip

\begin{enumerate} 
\item $\dpt \int_{T^{\dagger}}^t \dfrac{\mathrm{d}x(\tau)}{e^{x(\tau)} - 1 + \varepsilon} \geq - \int_{T^{\dagger}}^t  \mathcal{S}(\tau)\mathrm{d}\tau$

\medskip

\item $\dpt x(t) \geq \ln{\left( \dfrac{1-\varepsilon}{1-h_{\varepsilon}(t)} \right)}$
\end{enumerate}
\end{enumerate}
\end{lem}

\bigbreak
\begin{proof}
\begin{enumerate}
\item The proof runs along the same lines as Lemma \ref{lema12} (1st item).\\

\item The proof is a consequence of   the following chain of equivalences (for $t>T^\dagger $):
\begin{eqnarray}
\nonumber &&\int_{T^{\dagger}}^t \dfrac{\mathrm{d}x(\tau)}{e^{x(\tau)} - 1 + \varepsilon} \geq - \int_{T^{\dagger}}^t  \mathcal{S}(\tau)\mathrm{d}\tau \\
\nonumber \\
\nonumber &\Leftrightarrow&- \left[ \left( x(t) - x(T^{\dagger}) \right) - \left( \ln{| e^{x(t)} - 1 + \varepsilon |}   -  \ln{| e^{x(T^{\dagger})} - 1 + \varepsilon|} \right) \right] \geq - \int_{T^{\dagger}}^t  \mathcal{S}(\tau)\mathrm{d}\tau \\
\nonumber \\
\nonumber &\Leftrightarrow& e^{x(t)} - 1 + \varepsilon \geq \left( e^{x(T^{\dagger})} - 1 + \varepsilon \right) e^{x(t)} e^{-x(T^{\dagger})} e^{- \int_{T^{\dagger}}^t  \mathcal{S}(\tau)\mathrm{d}\tau} \\
\nonumber \\
\nonumber &\Leftrightarrow& 1 - \dfrac{ 1 - \varepsilon}{e^{x(t)}} \geq \left( e^{x(T^{\dagger})} - 1 + \varepsilon \right) \exp{\left\{-\left(    x(T^{\dagger}) + \int_{T^{\dagger}}^t  \mathcal{S}(\tau)\mathrm{d}\tau  \right)\right\}} \\
\nonumber \\
\nonumber &\Leftrightarrow& e^{-x(t)} \leq \dfrac{1- \left( e^{x(T^{\dagger})} - 1 + \varepsilon \right) \exp{\left\{-\left(    x(T^{\dagger}) + \int_{T^{\dagger}}^t  \mathcal{S}(\tau)\mathrm{d}\tau  \right)\right\}}}{1-\varepsilon} \\
\nonumber \\
\nonumber &\Leftrightarrow&  \dpt x(t) \geq \ln{\left( \dfrac{1-\varepsilon}{1-h_{\varepsilon}(t)} \right)}.
\end{eqnarray}
\end{enumerate}
\end{proof}

 For a fixed $\varepsilon>0$, combining Lemmas \ref{lema14A} and \ref{lema16A}, there exists $t>T^\dagger(\varepsilon)$ such that
$$
\ln{\left( \dfrac{1-\varepsilon}{1-h_{\varepsilon}(t)} \right)}\leq x(t)  \leq - \ln{\left( 1-h(t) \right)}.
$$
  Since 
$$
\lim_{t\rightarrow+\infty} \ln{\left( \dfrac{1-\varepsilon}{1-h_{\varepsilon}(t)} \right)}=\ln(1-\varepsilon) \qquad \text{and} \qquad   \lim_{t\rightarrow+\infty} - \ln{\left( 1-h(t) \right)}=0,
$$
  then, by the Squeezing Theorem, we get

$$\dpt \lim_{\varepsilon\rightarrow 0} \lim_{t\rightarrow +\infty} x(t)  =\dpt \lim_{\varepsilon\rightarrow 0} \ln(1-\varepsilon) =0.$$

  Using the change of variable \eqref{Change11}, we conclude that $S(t) \rightarrow \mathcal{S}$ as $t \rightarrow \infty$, in the topology of pointwise convergence, {\it i.e.}  
 ${\displaystyle \lim _{t\to+ \infty} \sup{\{dist(S(t),\mathcal{S}(t)): t > T^{\dagger} \}}=0}$.
The proof  of (3) of Theorem \ref{thA} is now complete.

\section{Proof of (4) of Theorem \ref{thA}}
\label{section_of_permanence}

In this section, we prove that system \eqref{modelo3} is permanent,  that is, there are positive constants $c_1, c_2$, $C_1, C_2,T_0>0$ such that for all solutions $\left( S(t),I(t) \right)$ with initial conditions
$S_0 > 0$, $I_0 > 0$, we have
$$ c_1 \leq S(t) \leq C_1 \quad \text{and} \quad c_2 \leq I(t) \leq C_2, $$ for all $t \geq T_0$.
Constants $C_1$ and $C_2$ come from Lemma \ref{region_}. Then we may take (for instance)
$$
C_1=C_2= \dfrac{A (\sigma + g + A)}{\sigma + g}>0.
$$

\noindent In the region {\Large \ding{175}} defined by $R_p>1$, the constant $c_1$ comes from the time average \eqref{Rp_eq}
$$1<\mathcal{R}_p=\dpt \dfrac{\beta_0}{\sigma + g}\dfrac{1}{T} \int_0^T  \mathcal{S}(t) \, \mathrm{d}t. $$
 
  We only need to prove that there exists a constant $c_2 > 0$ such that $I(t) \geq c_2$ for $t$ large enough. This assumption will be proved by contradiction, {\it i.e.\/} we assume that for all $c_2 > 0$, we have (see Figure \ref{case1case2}): \\

 \begin{enumerate}
 \item  $I(t) < c_2$ for all $t \geq t_1$ or \\
 \medskip
 \item  $I(t)<c_2$ at infinitely many subintervals of $[t_1, +\infty)$. Let $t^{\star} = \inf_{t > t_1}\{I(t) < c_2 \}$. There are two possibilities for such a $t^\star$:
 \smallskip
 \begin{enumerate}
 \item $t^\star= n_1 T$, for $n_1\in \NN$ and 
 \smallskip
 \item $t^\star \neq  n_1 T$, for $n_1\in \NN$.\\
 \end{enumerate}

 \begin{figure*}[ht!]
\includegraphics[width=0.96 \textwidth]{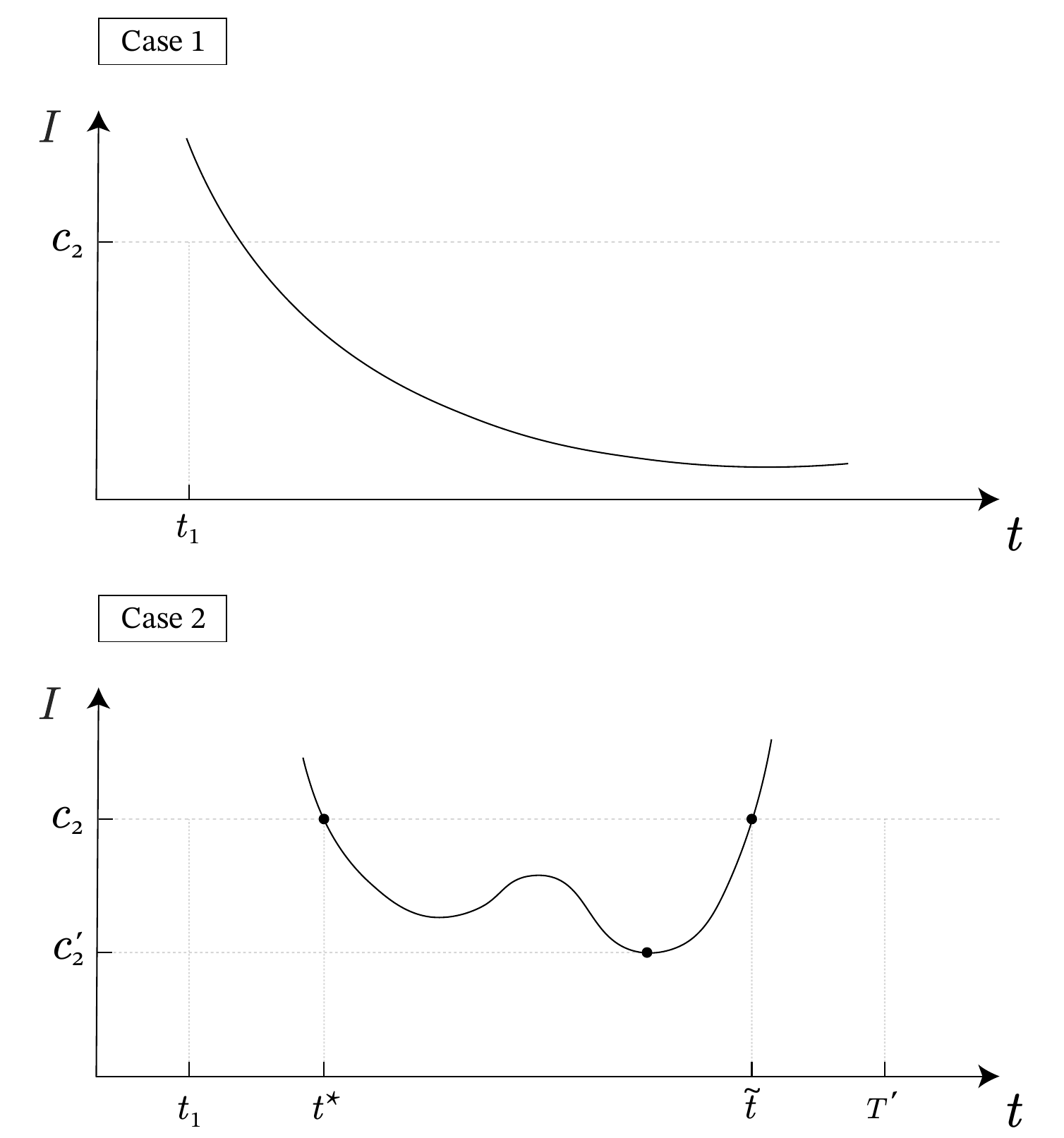}
\caption{\small Illustration of Cases 1 and  (2a).}
\label{case1case2}
\end{figure*}
\end{enumerate}

\noindent \textbf{Case (1)}  Assuming that $I(t) < c_2$ for all $t \geq t_1$, by including the $-\beta_0 c_2 S(t)<0$ in \eqref{s_wt_i} we get:
 \begin{equation*}
\label{permanence_model}
\begin{array}{lcl}
\begin{cases}
\bigskip
&\dot{S} \geq S(t)(A-S(t)) - \beta_0 c_2 S(t), \quad \,  \qquad t \neq nT  \\
&S(t^+) \geq (1-p) S(t), \qquad \qquad \quad \,\,\,\,\, \qquad t = nT .
\end{cases}
\end{array}
\end{equation*}

 Now, let $z(t)$ be the solution of the following system:
 \begin{equation}
\label{zt_equation}
\begin{array}{lcl}
\begin{cases}
\bigskip
&\dot{z}(t) = z(t)(A-z(t)) - \beta_0 c_2 z(t), \qquad \,  \qquad t \neq nT  \\
& z(t^+) = (1-p) z(t), \qquad \qquad \qquad \quad \,\,\,\,  \qquad t = nT .
\end{cases}
\end{array}
\end{equation}

\noindent Solving \eqref{zt_equation} as in \eqref{s_wt_i}, for $c_2>0$ arbitrarily small, we conclude that it has a periodic solution given explicitly by

\begin{eqnarray*}
\nonumber  \mathcal{Z}(t) &=& \dfrac{K}{1+\dfrac{p}{K \Big( (1-p) - e^{-KT} \Big)} K e^{-K(t-nT)}} = \dfrac{K \Big( (1-p) - e^{-KT} \Big)}{(1-p) - e^{-KT} + pe^{-K(t-nT)}} ,
\nonumber \\
\nonumber \\
\nonumber &=& \dfrac{K}{1+\dfrac{p}{ (1-p) - e^{-KT}}  e^{-K(t-nT)}} \\
\nonumber \\
\nonumber \\
\end{eqnarray*}
 where $ nT \leq t < (n+1)T$, $t>t_1$, $n\in \NN$ and 
$K \coloneqq -\beta_0 c_2 + A $.

  Now it is easy to conclude that $S(t) \geq z(t)$, $S(0^+) \geq z(0^+)$, and $z(t) \to \mathcal{Z}(t)$ in the pointwise convergence topology, where $\mathcal{Z}$ is the periodic solution of \eqref{zt_equation}.
Therefore, there exist $\varepsilon_1 > 0$ (arbitrarily small) and $t_2>t_1$ such that $S(t) \geq z(t) > \mathcal{Z}(t) - \varepsilon_1$; from \eqref{modelo3} and for $t > t_2$ we may write\\
\begin{equation}
\label{I_t_permanence}
\begin{array}{lcl}
\begin{cases}
 &\dot{I}(t) \geq \Big[ \beta_0 \left(\mathcal{Z}(t) - \varepsilon_1\right) - (\sigma + g) \Big] I(t), \quad \,  \qquad t \neq nT  \\
&I(nT^+) = I(nT).\\
\end{cases}
\end{array}
\end{equation}

 Let $N \in \NN$ and $NT \geq t_2$. Integrating \eqref{I_t_permanence} between $(nT, (n+1)T]$, $n \geq N$, we have
 $ I\big((n+1)T \big) \geq I(nT) \gamma $, 
where $$\displaystyle \gamma = \exp{ \left\{  \int_{nT}^{(n+1)T} \beta_0 \mathcal{Z}(t) \, \mathrm{d}t - \big( \beta_0 \varepsilon_1 + (\sigma + g) \big) T \right\}}.$$ \\ 
Using induction over $N\in \NN$ we get
\begin{equation}
\label{eqgama}
I\big((N+n)T \big) \geq I(NT) \, \gamma^n .
\end{equation}

\noindent \textbf{Claim:} If  $\mathcal{R}_p > 1$, then $\gamma>1$.
 

 \begin{proof} The proof follows from the chain of equivalences:
 \begin{eqnarray*}
\nonumber  &&\mathcal{R}_p > 1 \\
\nonumber \\
 \nonumber &\overset{\eqref{RpAND_Sc} \, \text{and}\,  \varepsilon_1 \gtrsim 0}{\Leftrightarrow}& \dfrac{\beta_0}{\sigma + g} \dfrac{1}{T} \int_{nT}^{(n+1)T} ( \mathcal{Z}(t) - \varepsilon_1) \, \mathrm{d}t > 1  \\
\nonumber \\
\nonumber &\Leftrightarrow& \exp{\Big\{  \int_{nT}^{(n+1)T} \beta_0 \mathcal{Z}(t) \, \mathrm{d}t - \big( \beta_0 \varepsilon_1 + (\sigma + g) \big) T \Big\}} > 1 \\
\nonumber \\
 \nonumber &\Leftrightarrow& \gamma > 1 .
\end{eqnarray*} 
\end{proof}

  Since  $\gamma > 1$, we see that $0<I(NT) \, \gamma^n \to \infty$ as $n \to \infty$, which  is a  contradiction by Lemma \ref{region_}. So, there exist $c_2,T_0 > 0$ such that $I(t) > c_2$ for all $t \geq T_0$. \\

\noindent \textbf{Case (2a)} Let $t^{\star} = n_1 T$, $n_1 \in \NN$. 
Then, $I(t) \geq c_2$ for $t \in [t_1, t^{\star}]$ and $I({t^{\star}}) = c_2$ and $I$ is decreasing. By Lemma \ref{infected_decreases}, we have
\begin{eqnarray}
\nonumber && \dot{I}(t) < 0  \Leftrightarrow \beta_0 S(t) < \sigma + g  \overset{0 < c_1 \leq S(t)}{\Rightarrow}  \beta_0 c_1< \sigma + g .
\end{eqnarray} 
\noindent Hence, we can choose $n_2, n_3 \in \NN$ such that $t^{\star} + n_2 T>t_1$ and 
\begin{eqnarray*}
\label{eq_4.22}
\gamma^{n_3}  \exp{\Big\{  (\beta_0 c_1 - \sigma - g) n_2 T \Big\}} > \gamma^{n_3}  \exp{\Big\{  (\beta_0 c_1 - \sigma - g) (n_2 + 1) T \Big\}} > 1 .
\end{eqnarray*} 
and define $T' = n_2 T + n_3 T$. \\

\noindent \textbf{Claim:} There exists $t_2 \in (t^{\star}, t^{\star} + T']$ such that $I(t_2) > c_2$.

\begin{proof} Let us assume that this is not true, {\it i.e.} that there is no such $t_2$.

Since $z(t) \to \mathcal{Z}(t)$ as $t \to \infty$ (in the pointwise convergence topology), as above, we may write $\mathcal{Z}(t) - \varepsilon_1 < z(t) \leq S(t)$, for $t^{\star} + n_2 T \leq t \leq t^{\star} + T'$. From \eqref{I_t_permanence} and for $t^{\star} + n_2 T \leq t \leq t^{\star} + T'$, we have (see equation \eqref{eqgama})

$$
I(t^{\star} + T') \geq I(t^{\star} + n_2 T) \gamma^{n_3}.
$$

From system \eqref{modelo3}, we get
\begin{equation}
\label{eq_4.23}
\begin{array}{lcl}
\begin{cases}
\bigskip
&\dot{I}(t) \geq \left(( \beta_0 c_1 - (\sigma + g) \right) I(t), \quad \quad \,\,\,\,  \qquad t \neq nT,  \\
&I(t^+) = I(t), \qquad \qquad \qquad \qquad \quad \,\,\, \qquad t = nT.
\end{cases}
\end{array}
\end{equation}
for $t \in [t^{\star}, t^{\star} + n_2 T]$. Integrating \eqref{eq_4.23} between $t^{\star}$ and $t^{\star} + n_2 T$, we have:

\begin{eqnarray*}
\label{eq_4.24}
\nonumber  \int_{t^{\star}}^{t^{\star} + n_2 T} \dfrac{\mathrm{d}I(t)}{I(t)} &\geq&  \int_{t^{\star}}^{t^{\star} + n_2 T}\big( \beta_0 c_1  - (\sigma + g) \big) \, \mathrm{d}t \\
\nonumber \\
\nonumber \Leftrightarrow I(t^{\star} + n_2T) &\geq& I(t^{\star}) \exp{\big\{\big(\beta_0 c_1 - (\sigma + g)\big) n_2 T \big\}} \\
\nonumber \\
\Leftrightarrow I(t^{\star} + n_2T) &\geq& c_2 \exp{\big\{\left[\beta_0 c_1 - (\sigma + g)\right]n_2 T \big\}} ,
\end{eqnarray*}
leading to
$$
I(t^{\star} + T') \geq c_2\exp{\big\{\left[\beta_0 c_1 - (\sigma + g)\right]n_2 T \big\}} \gamma^{n_3} > c_2 ,
$$
 which is a contradiction, since we have assumed that no $t \in (t^{\star}, t^{\star} + T']$ exists such that $I(t) > c_2$.\\
 \end{proof}

 Let $\tilde{t} = \inf_{t>t^{\star}}\{I(t) > c_2\}$.
Then, for $t \in (t^{\star}, \tilde{t})$, $I(t) \leq c_2$ and $I(\tilde{t}) = c_2$, since $I(t)$ is continuous and $I(t^+) = I(t)$ when $t = nT$. 
For $t \in (t^{\star}, \tilde{t})$, suppose $$t \in (t^{\star} + (k-1)T, t^{\star}+kT],\quad k \in \NN\quad  \text{and} \quad k \leq n_2 + n_3.$$ 
Therefore, from \eqref{eq_4.23} we have
\begin{eqnarray*}
\nonumber I(t) &\geq& I(t^{\star}) \exp{\big\{(k-1) \left(\beta_0 c_1 - (\sigma + g)\right) T \big\}} \exp{\big\{ \left(\beta_0 c_1 - (\sigma + g)\right)\left( t - (t^{\star} + (k-1)T) \right) \big\}}\\
\nonumber \\
\nonumber &\geq& c_2  \exp{\big\{k \left(\beta_0 c_1 - (\sigma + g)\right) T \big\}} \\
\nonumber \\
&\geq& c_2 \exp{\big\{(n_2 + n_3) \left(\beta_0 c_1 - (\sigma + g)\right) T \big\}} .\\
\end{eqnarray*}

  Let $c_2' = c_2 \exp{\{ (n_2 + n_3) (\beta_0 c_1 - (\sigma + g)) T \}}<c_2$.  This lower bound does not depend  neither  on $t^\star$ nor $\tilde{t}$.
Hence, we have $I(t) \geq c_2'$ for $t \in (t^{\star}, \tilde{t})$ and then for all $t>t_1$. For $t > \tilde{t}$, the same argument can be extended to $+\infty$ since $I(\tilde{t}) \geq c_2$ and $c_2'$ does not depend on the interval. This is a contradiction.
Hence, there exist $c_2,T_0 > 0$ such that $I(t) > c_2$ for all $t \geq T_0$. The proof of \textbf{Case (2b)} is entirely analogous and it is left to the reader. \\

 The Poincar\'e-Bendixson theorem for Impulsive planar flows \cite[Theorem 3.9]{Bonotto2008} says that the $\omega$--limit of \eqref{modelo3}  is  an equilibrium, a periodic solution or a union of saddles that are heteroclinically connected. 
The only equilibrium of  \eqref{modelo3} is the origin.  
In order to allow  endemic {\it $T$-periodic} solutions for $F_I$, then $ \dpt  \dfrac{1}{T} \int_0^T  S(t) \, \mathrm{d}t = S_c $. In this case, we know that there is an implicit (positive) map \,$\mathcal{I}\equiv \mathcal{I}(t)$, where $t\in \RR_0^+$.
Since there are no more candidates for $\omega$--limit sets, we have in the topological closure of the first quadrant the following limit sets:
$$
(0,0), \, \,   ( \mathcal{S}, 0) \, \, \text{and} \, \, (\mathcal{S}, \mathcal{I}).
$$
Since $(0,0)$ is repulsive (by Proposition \ref{stable_periodic_solution}), $( \mathcal{S}, 0)$ is a saddle (attracting in the first component and repelling in the second), the solutions should converge to the endemic periodic solution $(\mathcal{S}, \mathcal{I})$ whose explicit solution is not known.

\begin{figure*}[ht!]
\includegraphics[width=0.75 \textwidth]{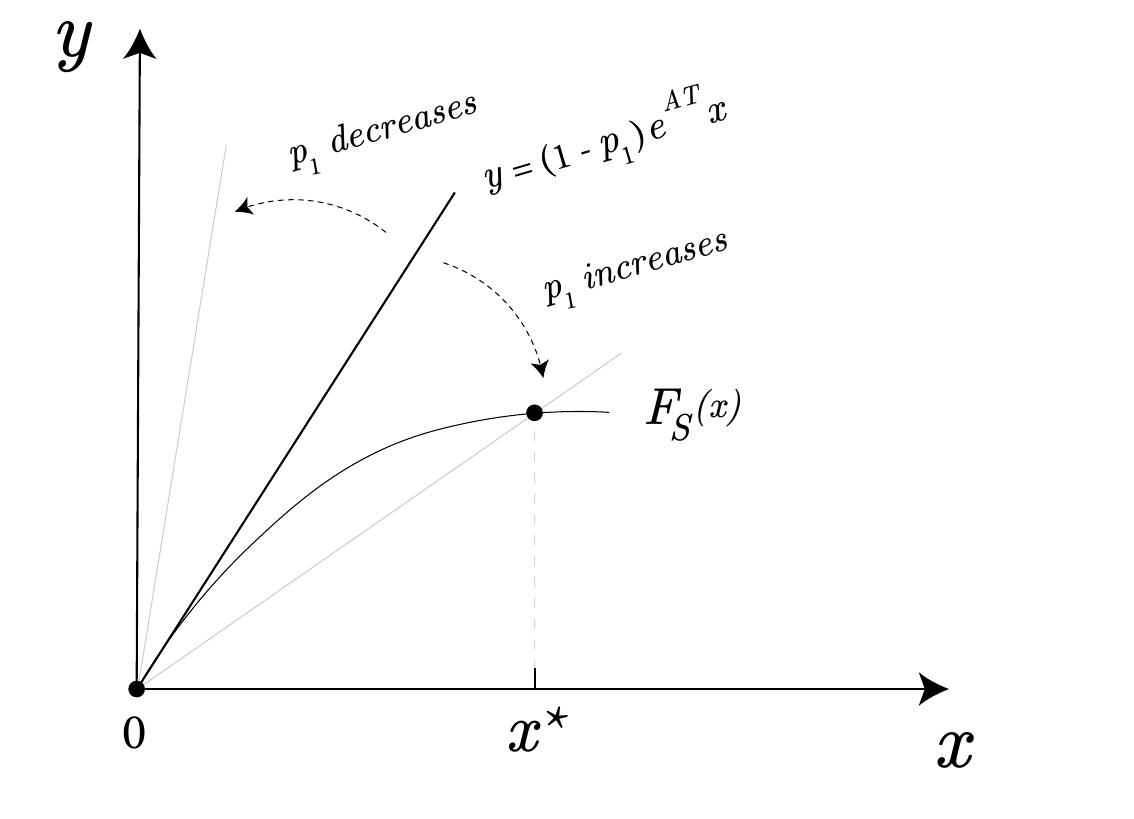}
\caption{\small Illustration of the existence of a degenerate saddle-node bifurcation at $p=p_1(T)$, for $T>0$ fixed.}
\label{fixed_x}
\end{figure*}

\section{Proof of (6) of Theorem \ref{thA}}
\label{section:bif}
For $x\geq 0$, it is easy to check that $\dpt F_S'(x)=\frac{P(x)}{Q^2(x)}$, where

$$
P(x)= A(1-p) e^{AT}(x(e^{AT}-1)+A)-Ax(1-p)e^{AT}(e^{AT}-1)
$$

\noindent and
$$
Q(x)= x(e^{AT}-1)+A.
$$

\noindent In particular, we have $F'_S(0)= (1-p)e^{AT}$. Since

$$
(1-p)e^{AT}=1 \quad \Leftrightarrow \quad p=p_1(T),
$$

\noindent then  we conclude that $p_1$ corresponds to a saddle-node bifurcation of $F_S$ at $x=0$. Indeed, if $p>p_1$, then the map $F_S$ does not have fixed points besides $x=0$; if $p<p_1$, then one extra fixed point emerges, as shown in Figure \ref{fixed_x}. In terms of $p_2$, if $p>p_2$, then the periodic solution $(\mathcal{S},0)$ changes its stability because the Floquet multiplier $\lambda_2$ (cf. \eqref{eigenvalue2}) crosses the unit circle; this is why we have a transcritical bifurcation at $p=p_2$.

\section{Proof of (7)  of Theorem \ref{thA}}
\label{s:proof(7)}
  Item (7) is proved in Lemmas \ref{Lema7} and \ref{e_qui_va_lence} of Section \ref{section:Rp}.

 \section{Proof of Corollary \ref{cor1}} \label{proof_cor_1}
 For $\gamma, \omega>0$  and $A>S_c$, there is a hyperbolic  {\it $T$-periodic} solution $(\mathcal{S}, 0)$ whose existence does not depend on $\beta_\gamma$. Therefore, it exists if $p<p_1(T)$ by Lemma \ref{estroboscopica_S}. However, the region of the phase space where it is stable depends on $\gamma$. Indeed,
\begin{eqnarray}
\nonumber  && \lambda_2<1\\
\nonumber && \\
\nonumber &\overset{\eqref{eigenvalue2}, \text{ adapted}}{\Leftrightarrow}& \exp{ \left\{  \dpt \int_0^T \beta_\gamma(t)\mathcal{S}(t) \, \mathrm{d}t - \left(\sigma + g \right)T \right\} } < 1 \\
\nonumber && \\
\nonumber  &\Leftrightarrow&  \dpt \int_0^T\beta_0 \left(1+\gamma \Psi (\omega t)\right) \mathcal{S}(t) \, \mathrm{d}t  - \left(\sigma + g \right)T < 0 \\
\nonumber && \\
\nonumber && \\
\nonumber  &\Leftrightarrow&  \dpt \frac{1}{\left(\sigma + g \right)T}\int_0^T\beta_0  \mathcal{S}(t) \, \mathrm{d}t + \frac{1}{\left(\sigma + g \right)T}\int_0^T    {\beta_0} \gamma \Psi (\omega t) \mathcal{S}(t) \, \mathrm{d}t < 1  \\
\nonumber  &\Leftrightarrow& \dpt \frac{\beta_0}{\left(\sigma + g \right)T}\int_0^T  \mathcal{S}(t) \, \mathrm{d}t + \frac{\beta_0}{\left(\sigma + g \right)T}\int_0^T \gamma \Psi (\omega t) \mathcal{S}(t) \, \mathrm{d}t < 1 \\
\nonumber \\
\nonumber  &\overset{\eqref{S_c(def)}}{\Leftrightarrow}&  \dpt \frac{1}{S_c T}\int_0^T  \mathcal{S}(t) \, \mathrm{d}t + \frac{1}{S_c T}\int_0^T \gamma \Psi (\omega t) \mathcal{S}(t) \, \mathrm{d}t < 1 \\
\nonumber  &\Leftrightarrow& \dpt \left[\ln{(1-p)} + AT \right] + \gamma \int_0^T  \Psi (\omega t) \mathcal{S}(t) \, \mathrm{d}t < S_c T \\
\nonumber \\
\nonumber  &\Leftrightarrow& p> 1 - \exp{\left\{- \left[(A - S_c)T + \gamma \int_0^T  \Psi (\omega t) \mathcal{S}(t) \, \mathrm{d}t \right]\right\}} \\
\nonumber && \\
\nonumber  &\Leftrightarrow& p>p_2^{\text{seas}}(T) .
\end{eqnarray}
 
\section{Proof of Theorem \ref{th: mainB}}
\label{SA_lab}
 The proof follows from   the item (4) of  Theorem \ref{thA} combined with  \cite{Anishchenko1993, Anishchenko2007} and Section 3.2 of  \cite{WangYoung2003}(\footnote{See the comment before Theorem 2 of \cite{WangYoung2003}, where the authors refer impulsive differential equations.}),    taking into account the following considerations:

\begin{itemize}
\item $\omega \in \RR^+$  plays the role of ``shear'' of \cite{WangYoung2003};  
\medskip
\item the term $\beta_\gamma(t) = \beta_0\left(1+\gamma \Psi(\omega t)\right) > 0$ may be seen as the radial ``kick'' of \eqref{modelo2} that is periodic in time; 
\medskip
\item  the non-autonomous periodic forcing of \eqref{modelo2}  has nondegenerate critical points (by \textbf{(C4)}) -- this avoids constant maps;
\medskip
\item the endemic periodic solution $(\mathcal{S}, \mathcal{I})$ of (4) of Theorem \ref{thA} is globally attracting  for \eqref{modelo3}.\\
\end{itemize}

For $\gamma>0$,  equation \eqref{modeloSIR} is equivalent to  

 \begin{equation}
 \begin{array}{lcl}
 \label{gamma>0}
\dot{x} = f_\gamma(x) \quad \Leftrightarrow \quad
\begin{cases}
\medskip
&\dot{S} = S(A-S) - \beta_\gamma {(t)} I S  \\
\medskip
&\dot{I} =  \beta_\gamma {(t)} IS - \left(\sigma+g\right) I  ,  \qquad \qquad \qquad t \neq nT  \\
\bigskip
&\dot{\theta} = \omega \pmod{2\pi}\\
\medskip
&S(nT) = (1-p) S(nT^-)  \\
\medskip
&I(nT) = I(nT^-)
\end{cases}
\end{array}
\end{equation}

\medskip

\noindent whose flow lies in $(\RR_0^+)^2 \times \mathbf{S}^1$, where $\mathbf{S}^1\equiv \RR \mod{\pi/\omega}$.  Since the kicks are radial, they do not affect the $\theta$--coordinate. The following argument follows from \cite{Anishchenko1993, Anishchenko2007} and Section 3.2 of  \cite{WangYoung2002}.

Using item (4) of  Theorem \ref{thA} applied to the amplitude component of \eqref{gamma>0}, one knows that the {\it $\omega$-limit} of Lebesgue almost all solutions of \eqref{gamma>0} is a strict subset of a two-dimensional attracting torus $\mathcal{T}_0$ (normally hyperbolic manifold). The torus is  ergodic when $T$ is not commensurable with $\omega/\tau$. Furthermore, for a dense set $\Delta$ of pairs $(T, \omega)\in [T_2, +\infty)\times \RR^+$ such that $T=k\tau/\omega$, one knows that this normally hyperbolic manifold is foliated by periodic solutions, as depicted on the  left-hand side of Figure \ref{XY_im}. In the terminology of Herman  \cite{Herman1977}, the set $\Delta$ corresponds to orbits with rational rotation number. 

Let $\Sigma$ be a cross-section to $\mathcal{T}_0$ where the first return is well defined.  It is easy to observe that the set $\Sigma \cap \mathcal{T}_0$ is part of a circle. If $(T, \omega) \in \Delta$ and $\gamma=0$,   there is at least one pair of periodic solutions  in $\mathcal{T}_0 $ which are heteroclinically connected, say $q_1$  (saddle) and $q_2$  (sink). The rotation number remains  constant within a  synchronization zone (valid for $\gamma>0$), also known as Arnold's tongue  \cite{Shilnikov2Turaev2004}. 

Since $\mathcal{T}_0$ is a normally hyperbolic manifold, then for $\gamma\gtrsim 0$, there is a manifold $\mathcal{T}_\gamma$ diffeomorphic to $\mathcal{T}_0$, which is still attracting; as $\gamma$  increases  further, the set $\Sigma \cap \mathcal{T}_\gamma$  becomes ``non-smooth'' and three generic\footnote{Valid in a residual set within the set of one-parameter families $(f_\gamma)_\gamma$   periodically perturbed by maps satisfying \textbf{(C4).}} scenarios may occur:
\begin{itemize}
\item[(A)]  the circle  loses its smoothness (near $q_2$) when a pair of multipliers of the cycle becomes complex (non-real) or one of the multipliers is negative. At the moment of bifurcation, the length of an invariant circle becomes infinite and the torus is destroyed. The transition to chaos can come either from a period-doubling bifurcation cascade or via the breakdown of a torus occurring near $q_2$  (route A of \cite[pp.124]{Anishchenko2007});
\item[(B)] homo or heteroclinic tangle of the dissipative saddle $q_1$ (route B of \cite[pp.124]{Anishchenko2007});
\item[(C)] distortion of the unstable manifold near a non-hyperbolic saddle-node; the torus becomes non-smooth  (route C of \cite[pp.124]{Anishchenko2007}).
\end{itemize}

Following \cite{Anishchenko2007}, the three mechanisms of torus destruction leads to a (non-hyperbolic) topological horseshoe-type map with a smooth bend.  They do not cause the absorbing area to change abruptly and thus represent the bifurcation mechanism of a soft transition to chaos. 
The torus destruction line in the two control parameter plane $(\omega, \gamma)$ is characterised by a complex structure \cite{Anishchenko1993, Anishchenko2007, Rodrigues_2022_JDDE}.  There are small regions (in terms of measure)  inside the resonance wedges where chaotic trajectories are observable: they correspond to strange attractors of H\'enon type and are associated with historic behavior. Other stable points with large period exist  as a consequence of the Newhouse phenomena. 
For a better understanding of Route (B), see Figure \ref{XY_im}. Compare also with Figure 2.10 of \cite{Anishchenko2007}.

\begin{figure*}[ht!]
\includegraphics[width=0.55 \textwidth]{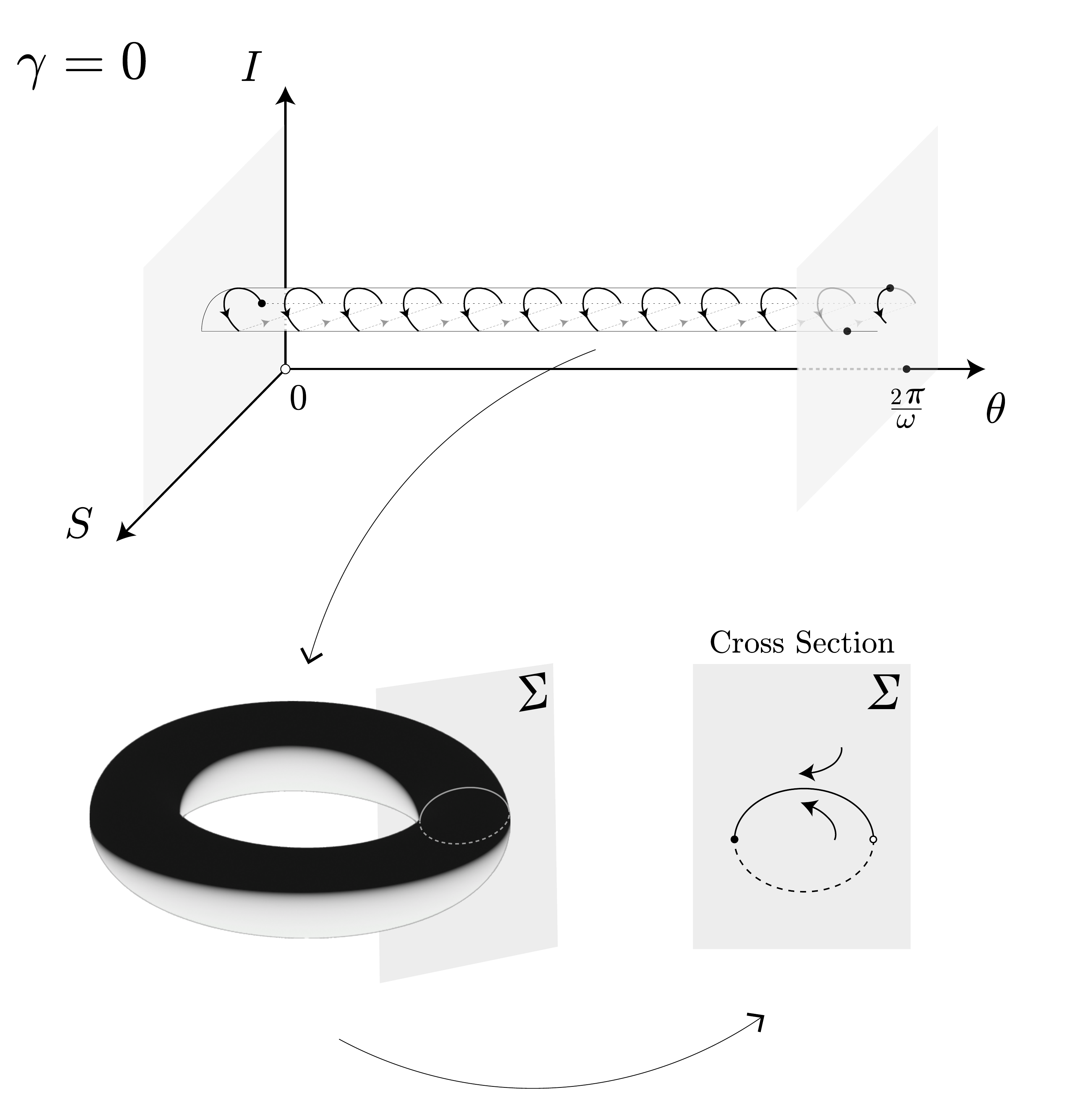}\includegraphics[width=0.50 \textwidth]{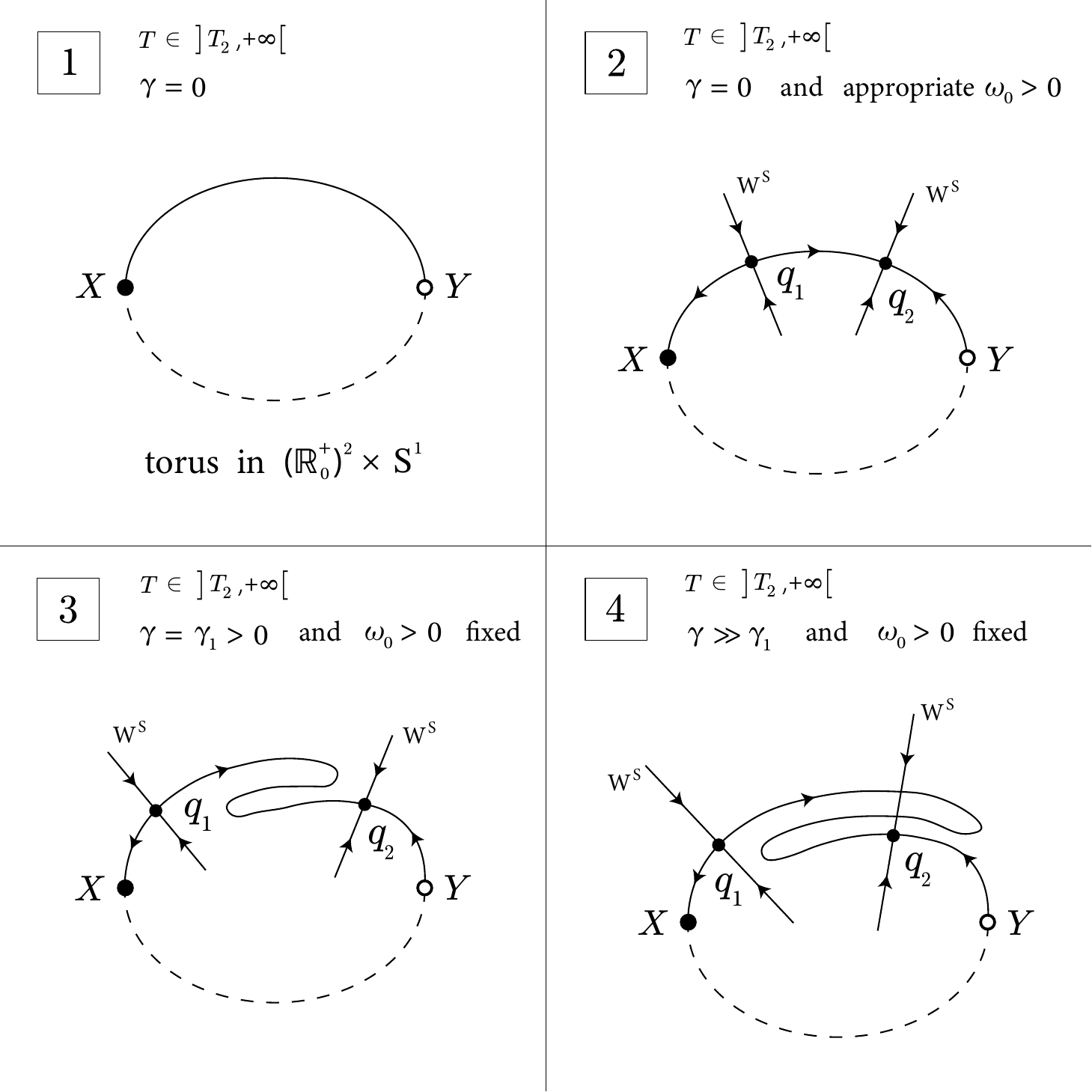}
\caption{\small {Sketch for the proof of Theorem \ref{th: mainB}: \bf (1).} Torus bifurcation for $\gamma = 0$. {\bf (2).} For $\gamma = 0$ and an appropriate $\omega \in \RR^+$, we observe the emergence of periodic solutions represented by $q_1$ and $q_2$ on $\mathcal{T}_0$. {\bf (3 and 4).} As $\gamma$ increases, the intersection $\Sigma \cap \mathcal{T}_0$ starts to lose its smoothness. In {\bf (4)}, the unstable manifold $W^u(q_1)$ intersects the stable manifold $W^s(q_2)$, leading to the formation of homoclinic tangles that indicate the emergence of horseshoes (chaos) -- compare with Figure 2.10 of \cite{Anishchenko2007}.}
\label{XY_im}
\end{figure*}


\section{Numerics} \label{s:numerics}

In this section we perform  some numerical simulations to illustrate the contents of  Theorems \ref{thA} and  \ref{th: mainB}. All simulations were performed via MATLAB\_R2018a software.

\medbreak  

\textbf{Figure \ref{art_}: }Numerical simulation of the five scenarios of system \eqref{modelo3} illustrated in the {bifurcation diagram $(T,p)$} in Figure \ref{Bif_Diag_ThA} with initial condition $(S_0, I_0) = (0.5, 0.4)$ described by Theorem \ref{thA}. The figure shows the behavior of the {\it Susceptible} and {\it Infectious} over time for different values of $p$ and corresponding phase space. Parameter values: $A=1$, $\beta_0 = 0.9$, $\sigma=0.2$, $g=0.5$, $T=4$ and $p \in [0,1]$. The green line represents the disease-free non-trivial periodic solution $(\mathcal{S},0)$, and the red line represents the endemic periodic solution $(\mathcal{S},\mathcal{I})$.
 
\medbreak  
 
\textbf{Figure \ref{SI_R3}:} Projection of the solution of \eqref{gamma>0}  with initial condition $(S_0, I_0, \theta_0) = (0.5, 0.4, 0)$ for different values of $\gamma$. The parameter values are strategically chosen in order to have the dynamics of Region   {\Large \ding{174}} of Figure \ref{Bif_Diag_ThA}: $A = 1$, $\beta_0 = 0.2$, $\sigma = 0.05$, $g = 0.02$, $T = 1$, $\omega = 0.1$ and $p = 0.5$. The dynamics of $S$ and $I$ show that the system tends towards the disease-free periodic solution $(\mathcal{S},0)$. As the amplitude of the seasonal variation $\gamma$ increases, the disease-free periodic solution starts to deform.
 
\medbreak
 
\textbf{Figure \ref{SItheta_R3}:}   Numerical simulation of the phase space $(S,I)$, $(S,I,\theta)$, and respective cross-section ($\theta = 2$) for three different values of $\gamma$ of \eqref{gamma>0}. Parameter values used in this simulation: $A = 1$, $T = 4$, $\beta_0 = 2$, $\omega = 6$, $\sigma = 0.2$, $g = 0.5$ and $p = 0.4$ with initial condition $(S_0, I_0, \theta_0) = (0.4074, 0.2645, 0)$. This simulation corresponds to Region  {\Large \ding{175}}, where $\mathcal{R}_p > 1$. For this initial condition and $\gamma=4.89$, there is a positive Lyapunov exponent (0.13).

\begin{figure*}[ht!]
\includegraphics[width=1.12 \textwidth]{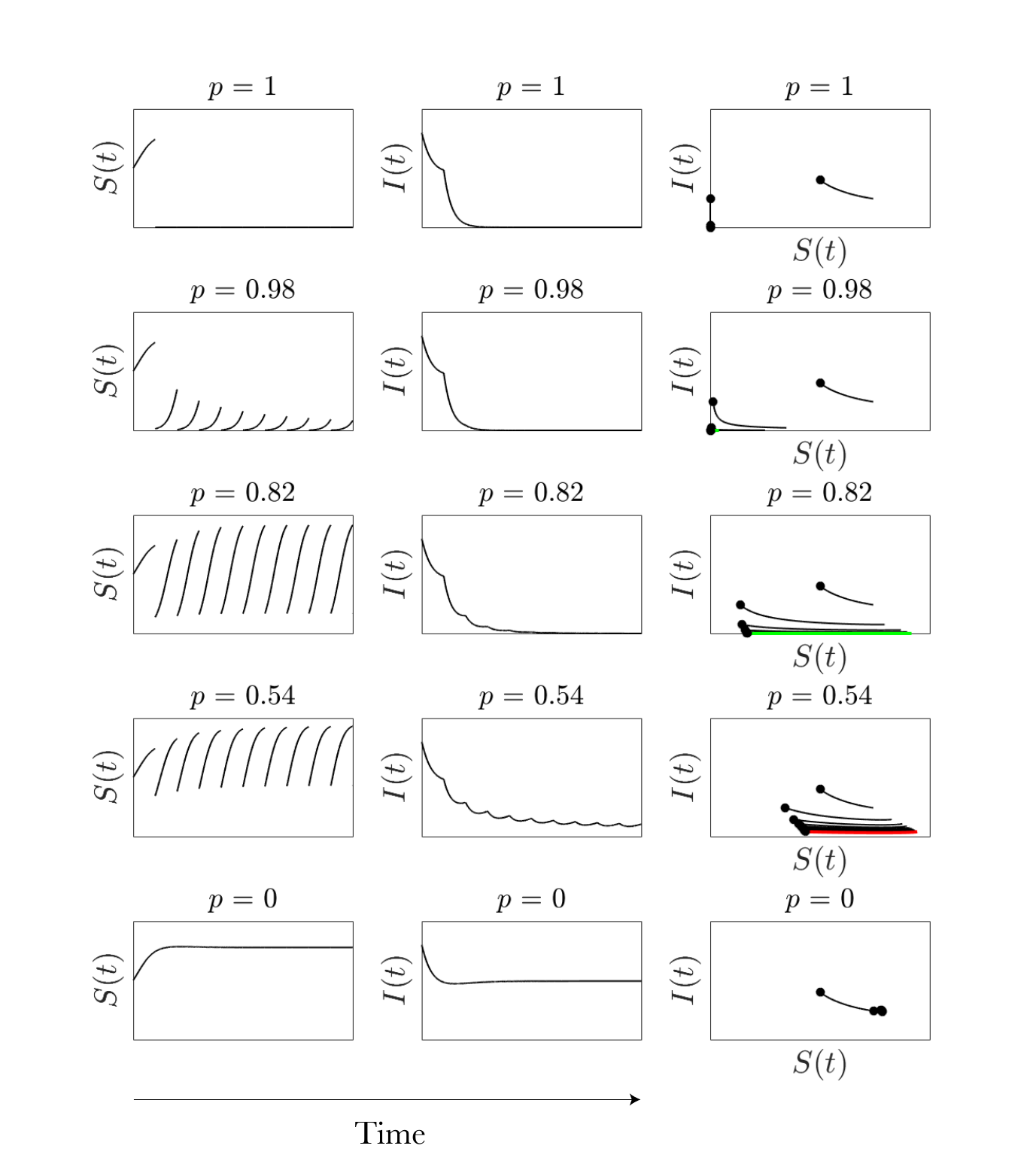}
\caption{\small Numerical simulation of the five scenarios of system \eqref{modelo3} illustrated in the {bifurcation diagram $(T,p)$} of Figure \ref{Bif_Diag_ThA} with initial condition $(S_0, I_0) = (0.5, 0.4)$. The figure shows the behavior of the {\it Susceptible} and {\it Infectious} over time for different values of $p$ and corresponding phase space. Parameter values: $A=1$, $\beta_0 = 0.9$, $\sigma=0.2$, $g=0.5$, $T=4$ and $p \in [0,1]$. The green line represents the disease-free non-trivial periodic solution $(\mathcal{S},0)$, and the red line represents the endemic periodic solution $(\mathcal{S},\mathcal{I})$.}
\label{art_}
\end{figure*}

\begin{figure*}[ht!]
\includegraphics[width=0.98 \textwidth]{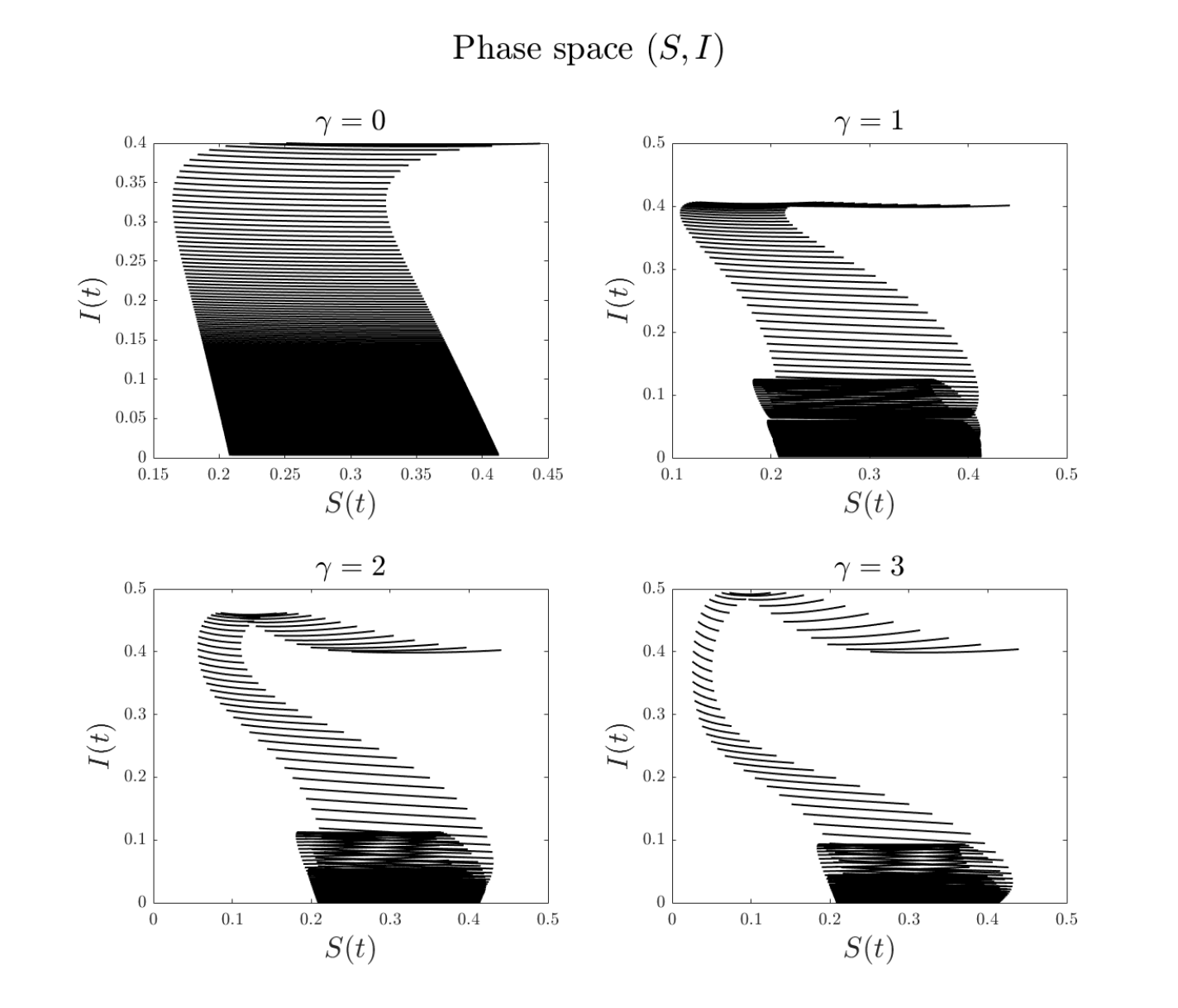}
\caption{\small  Projection of the solution of \eqref{gamma>0} with initial condition $(S_0, I_0, \theta_0) = (0.5, 0.4, 0)$ for different values of $\gamma$. The parameter values are strategically chosen in order to have the dynamics of Region   {\Large \ding{174}} of Figure \ref{Bif_Diag_ThA}: $A = 1$, $\beta_0 = 0.2$, $\sigma = 0.05$, $g = 0.02$, $T = 1$, $\omega = 0.1$ and $p = 0.5$. The dynamics of $S$ and $I$ show that the system tends towards the disease-free periodic solution $(\mathcal{S},0)$.} 
\label{SI_R3}
\end{figure*}

\begin{figure*}[ht!]
\includegraphics[width=1\textwidth]{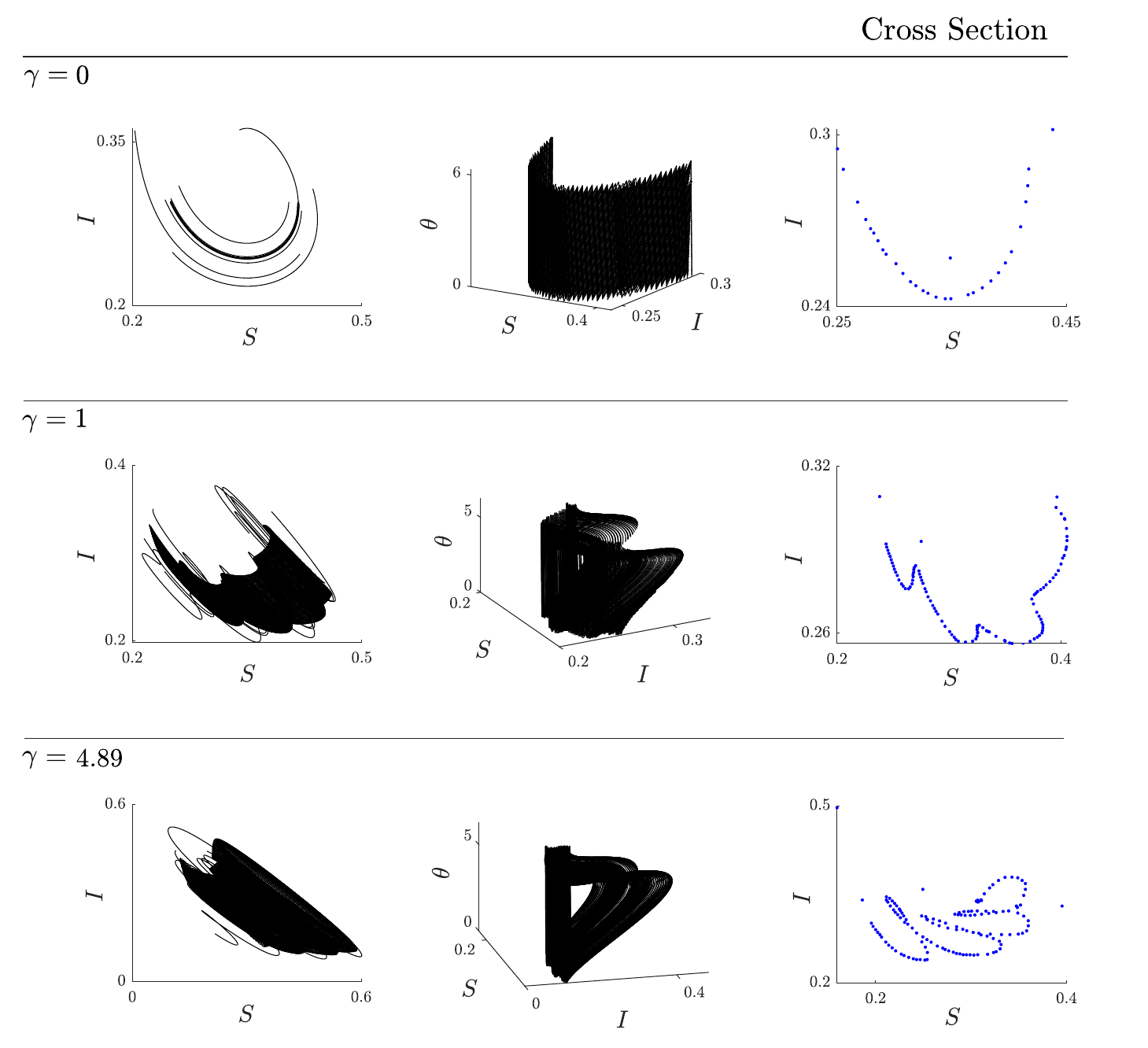}
\caption{\small Numerical simulation of the phase space $(S,I)$, $(S,I,\theta)$, and respective cross-section ($\theta = 2$) for three different values of $\gamma$ of \eqref{gamma>0}. Parameter values used in this simulation: $A = 1$, $T = 4$, $\beta_0 = 2$, $\omega = 6$, $\sigma = 0.2$, $g = 0.5$ and $p = 0.4$ with initial condition $(S_0, I_0, \theta_0) = (0.4074, 0.2645, 0)$. This simulation corresponds to Region  {\Large \ding{175}}, where $\mathcal{R}_p > 1$. 
The system exhibits chaos if $\gamma=4.89$ with a positive Lyapunov exponent  ($0.13$).}
\label{SItheta_R3}
\end{figure*}

\section{Discussion and Final Remarks} \label{s:section}

In this work, we have studied a modified SIR model, introducing logistic growth to the population of the {\it Susceptible} individuals and incorporating pulse vaccination ({\it Susceptible} individuals are $T$-periodically vaccinated),  conferring  immunity to a proportion $p$ of the {\it Susceptible} population. Additionally,  we have explored the model with and without seasonality   into the disease transmission rate. 

\subsection*{Results}

Regarding our first main result (Theorem \ref{thA}), in the absence of seasonality in the disease transmission rate ($\gamma=0$), the analysis of the model reveals five stationary scenarios, depicted in the $(T,p)$ bifurcation diagram of Figures \ref{Bif_Diag_ThA} and \ref{art_}.  The non-trivial periodic $\omega$-limit sets have the same period as the initial pulse.

For $p=0$  and $\mathcal{R}_0> 1$, the outcomes align with those of  \cite{CarvalhoRodrigues2023}: the {\it $\omega$-limit} of  all solutions with initial condition $S_0, I_0>0$ is the endemic equilibrium. For $T>0$ fixed and values of $p$ between $0$ and $p_2(T)$, our model is {\it permanent}. In this case, the associated {\it basic reproduction number} is greater than 1. This  agrees well with the findings of  \cite{Meng2008, Yongzhen2011, ZouGaoZhong2009}. Our contribution goes further in this direction --  we have proved the existence of a  globally stable  endemic {\it $T$-periodic} solution.

For $p\in (p_2(T), p_1(T))$, system   \eqref{modelo3} exhibits a disease-free non-trivial periodic solution, where the curves $p=p_1(T)$ and $p=p_2(T)$ correspond to  saddle-node and transcritical bifurcations, respectively.  

Under the effect of  seasonality in the disease transmission rate, if  $p<p_2(T)$, \,$k \tau/\omega = T$ (within a resonance wedge), and $\gamma \gg 1$ (large), Theorem \ref{th: mainB} indicates that the flow of \eqref{modeloSIR} exhibits a suspended topological horseshoe. Consequently, the number of {\it Infectious} individuals persists and is more difficult to control the disease. The proof of this result follows the same lines of  \cite{Anishchenko2007}.  Our results stress that insufficient vaccination coverage combined with seasonality can generate chaotic dynamics.

We have presented  numerical simulations to demonstrate that the periodic nature of pulse vaccination combined with seasonality might spread of the disease. Corollary \ref{cor1} and Figure \ref{p1p2SAZ} warn that neglecting the effect of the amplitude of the seasonal transmission  could cause an over-optimistic approximation of  the optimal pulse period \cite{Choisy2006}.

\subsection*{Literature}

In a recent study of a SIR model with pulse vaccination, the author of \cite{Herrera2023} pointed out two mechanisms that lead a classical SIR model to chaotic dynamics with low vaccine coverage. The first mechanism arises from the interaction of low birth and death rates combined with a high   contact rate between individuals. The second mechanism results from high birth and low contact rates.

 In our paper, we have analytically proved the emergence of chaos through a different technique: modulating the contact rate of the disease through a periodic function (seasonality), system \eqref{modeloSIR} may exhibit chaos under generic assumptions. Mathematically, the author of \cite{Herrera2023} proved chaos via the Zanolin's method \cite{Zanolin2010} (stretching rectangles along paths); in our paper, chaos emerges via Torus-breakdown theory (\cite{Anishchenko1993} and \cite[\S 2.1.4]{Anishchenko2007}).

\subsection*{Pulse vs. constant vaccination}
We compare the findings of \cite{CarvalhoRodrigues2023, Shulgin1998} with the present work.
Pulse vaccination requires a smaller percentage of people vaccinated to prevent a possible epidemic outbreak. 
 Since pulse vaccination is periodic, it is enough to vaccinate a proportion $p$ of the population to keep the number of {\it Susceptible} individuals below prescribed epidemic limits. On the other hand, constant vaccination requires a high rate of {\it Susceptible} individuals to be vaccinated to avoid spreading the disease.

 While constant vaccination offers constant and stable protection to the population, pulse vaccination can lead to poorly planned epidemic outbreaks, especially if the number of {\it Susceptible} individuals exceeds epidemic limits.

\subsection*{Open problem}
Since $C^2$--hyperbolic horseshoes  have zero Lebesgue measure, it is  possible for a map to have a horseshoe, and at the same time, the orbit of Lebesgue-almost every point tends to a sink. 
 Adapting the proof of \cite{WangYoung2003} we conjecture that (concerning system \eqref{modelo2}) if $\omega$ is large enough, then

\smallskip

\begin{equation*}
\label{abundance}
\liminf_{\varepsilon\rightarrow 0^+}\, \,  \frac{ {\rm Leb} \left\{\gamma \in [0,\varepsilon]: {f}_\gamma  \text{  exhibits a   strange attractor in the sense of  \cite{WangYoung2003}}\right\}}{\varepsilon}  >0,
\end{equation*}

\medskip

\noindent where  {Leb} denotes the one-dimensional Lebesgue measure.  
   Within the strange attractor, orbits jump around in a seemingly random fashion, and the future of individual orbits appears entirely unpredictable. For these chaotic attractors, however, there are laws of statistics that control the asymptotic distributions of Lebesgue almost all orbits in the basin of attraction of the attractor.

\subsection*{Final remarks and future work}
Our study suggests that the effective strategy for controlling an infectious disease governed by \eqref{modeloSIR} must be in a way that the proportion $p$ of vaccination is at the target level needed for the disease eradication, and the time $T$ between shots must be appropriate. The results suggest that maintaining a high proportion $p$ of vaccination and optimizing the period $T$  between two shots   shall increase the effectiveness of the vaccination strategy.
The inclusion of seasonality has dramatic impacts on the dynamics. 

One of the main challenges of the World Health Organization is the global eradication of certain diseases in countries with low incomes and high birth rates. From an applied perspective, our results could be  helpful for the optimal design of vaccination programs. 

Further refinements of the pulse vaccination strategy need to take seasonal disease dynamics into full account.  From a theoretical perspective, there is a need to develop analytical epidemiological models that incorporate information such as the natural period of the disease under analysis (resonance dynamics of \cite{Choisy2006}).

 Some lines of research need to be tackled such as the dependence of vaccination efficacy over time, multiresolution modeling and the availability of medical resources \cite{Bajiya2022}. Cross-immunity and delayed vaccination are   important factors to be considered  \cite{Goel2023}. In addition, strategies such as preventive vaccination and the impact of the media should also be explored \cite{Kumar2021, Kumar2022}. Finally, pulse vaccination modelling can be extended to include multi-group epidemic models, providing a more detailed and realistic analysis of disease dynamics in distinct subpopulations. These  directions could significantly improve our understanding of the epidemiological models.

\section*{Acknowledgements}

The authors are grateful to the anonymous referees, whose attentive reading and useful comments improved the final version of the article.


\begin{thebibliography}{777}
\bibliographystyle{unsrt}

\bibitem{Hethcote2000} 
Hethcote HW (2000) The Mathematics of Infectious Diseases. SIAM Rev. 42: 599--653

\bibitem{Cobey2020} 
Cobey S (2020) Modeling infectious disease dynamics. Science 368:713--714 

\bibitem{KermackMcKendrick1932} 
Kermack WO, McKendrick AG (1932) Contributions to the mathematical theory of epidemics. II. The problem of endemicity. Proc. R. Soc. Lond. 138: 55--83

\bibitem{BrauerChavez2012}  
Brauer F, Castillo-Chavez C (2012). Mathematical Models in Population Biology and Epidemiology. New York. Springer

\bibitem{Diekmann2012} 
Diekmann O, Heesterbeek JA, Britton T (2012) Mathematical Tools for Understanding Infectious Disease Dynamics. Princeton University Press.

\bibitem{DriesscheWatmough2002}  
van den Driessche P, Watmough J (2002) Reproduction numbers and sub-threshold endemic equilibria for compartmental models of disease transmission. Math. Biosci. 180: 29--48

\bibitem{CarvalhoPinto2021}  
Maur\'icio de Carvalho JPS, Moreira-Pinto B (2021) A fractional-order model for CoViD-19 dynamics with reinfection and the importance of quarantine. Chaos Solitons Fractals 151: 111275 

\bibitem{Plotkin2005}  
Plotkin SA (2005) Vaccines: past, present and future. Nat. Med. 11: S5--S11 

\bibitem{Plotkin2011} 
Plotkin SA, Plotkin SL (2011) The development of vaccines: how the past led to the future. Nat. Rev. Microbiol. 9: 889--893 

\bibitem{CarvalhoRodrigues2023} 
Maur\'icio de Carvalho JPS, Rodrigues AA (2023) SIR Model with Vaccination: Bifurcation Analysis. Qual. Theory Dyn. Syst. 22: 32 pages

\bibitem{Makinde2007} 
Makinde OD (2007) A domian decomposition approach to a SIR epidemic model with constant vaccination strategy. Appl. Math. Comput. 184: 842--848  

\bibitem{Shulgin1998} 
Shulgin B, Stone L, Agur Z (1998) Pulse vaccination strategy in the SIR epidemic model. Bull. Math. Biol. 60: 1123--1148 

\bibitem{Agur1993} 
Agur Z, Cojocaru L, Mazor G, Anderson RM, Danon YL (1993) Pulse mass measles vaccination across age cohorts. Proc. Natl. Acad. Sci. USA. 90: 11698--11702

\bibitem{Shulgin2000} 
Stone L, Shulgin B, Agur Z (2000) Theoretical examination of the pulse vaccination policy in the SIR epidemic model. Math. Comput. Model. 31: 207--215 

\bibitem{LuChiChen2002} 
Lu Z, Chi X, Chen L (2002) The Effect of Constant and Pulse Vaccination on SIR Epidemic Model with Horizontal and Vertical Transmission. Math. Comput. Model. 36: 1039-1057


\bibitem{Meng2008} 
Meng X, Chen L (2008) The dynamics of a new SIR epidemic model concerning pulse vaccination strategy. App. Math. Comput. 197: 582--597

\bibitem{JinHaque2007} 
Jin Z, Haque M (2007) The SIS Epidemic Model with Impulsive Effects. Eighth ACIS International Conference on Software Engineering, Artificial Intelligence, Networking, and Parallel/Distributed Computing (SNPD 2007), Qingdao, China, 2007, pp. 505--507 

\bibitem{Kanchanarat2023} 
 Kanchanarat S, Nudee K, Chinviriyasit S, Chinviriyasit W (2023) Mathematical analysis of pulse vaccination in controlling the dynamics of measles transmission. Infect. Dis. Model. 8: 964--979


\bibitem{Buonomo2018} 
Buonomo B, Chitnis N, d'Onofrio A (2018) Seasonality in epidemic models: a literature review. Ric. Di Mat. 67: 7--25

\bibitem{CarvalhoRodrigues2022} 
Maur\'icio de Carvalho JPS, Rodrigues AAP (2022) Strange attractors in a dynamical system inspired by a seasonally forced SIR model. Phys. D 434: 12 pages

\bibitem{Keeling2001} 
Keeling MJ, Rohani P, Grenfell BT (2001) Seasonally forced disease dynamics explored as switching between attractors, Physica D. 148: 317--335

\bibitem{Barrientos2017} 
Barrientos PG, Rodr\'iguez JA, Ruiz-Herrera A (2017) Chaotic dynamics in the seasonally forced SIR epidemic model. J. Math. Biol. 75: 1655--1668

\bibitem{DuarteJanuario2019} 
Duarte J, Janu\'ario C, Martins N, Rogovchenko S, Rogovchenko Y (2019) Chaos analysis and explicit series solutions to the seasonally forced SIR epidemic model. J. Math. Biol. 78: 2235--2258

\bibitem{Wang2015} 
Wang L (2015) Existence of periodic solutions of seasonally forced SIR models with impulse vaccination. Taiwan. J. Math. 19: 1713--1729

\bibitem{Feltrin2018} 
Feltrin G (2018). Mawhin's Coincidence Degree. In: Positive Solutions to Indefinite Problems. Frontiers in Mathematics. Birkh\"auser, Cham.

\bibitem{Wang2020} 
Wang L (2020) Existence of Periodic Solutions of Seasonally Forced SEIR Models with Pulse Vaccination. Discrete Dyn. Nat. Soc. 2020: 11 pages

\bibitem{Jodar2008} 
J\'odar L, Villanueva RJ, Arenas A (2008) Modeling the spread of seasonal epidemiological diseases: Theory and applications

\bibitem{ZuWang2015} 
Zu J, Wang L (2015) Periodic solutions for a seasonally forced SIR model with impact of media coverage. Adv. Differ. Equ. 2015: 10 pages

\bibitem{IbrahimDenes2023} 
 Ibrahim MA, D\'enes A (2023) Stability and Threshold Dynamics in a Seasonal Mathematical Model for Measles Outbreaks with Double-Dose Vaccination. Mathematics 11: 20 pages


\bibitem{Guan2024}  
 Guan G, Guo Z, Xiao Y (2024) Dynamical behaviors of a network-based SIR epidemic model with saturated incidence and pulse vaccination. Commun. Nonlinear Sci. Numer. Simul. 137: 18 pages



\bibitem{Dishliev} 
Dishliev A, Dishlieva K, Nenov S (2012) Specific asymptotic properties of the solutions of impulsive differential equations. Methods and applications. Academic Publication, 2012

\bibitem{Lakshmikantham_livro} 
Lakshmikantham V, Bainov DD, Simeonov PS (1989) Theory of Impulsive Differential Equations. Series in Modern Applied Mathematics: Vol. 6. World Scientific

\bibitem{Agarwal_livro} 
Agarwal R, Snezhana H, O'Regan D (2017) Non-instantaneous impulses in differential equations Springer, Cham. 1--72, 2017

\bibitem{Bainov_livro} 
Bainov D, Simeonov P (1993) Impulsive differential equations: periodic solutions and applications. Vol. 66. CRC Press, 1993

\bibitem{Milev1990} 
Milev N, Bainov D (1990) Stability of linear impulsive differential equations. Comput. Math. Appl. 21: 2217--2224

\bibitem{SimeonovBainov1988} 
Simeonov P, Bainov D (1988) Orbital stability of periodic solutions of autonomous systems with impulse effect.
Int. J. Systems Sci. 19: 2561--2585

\bibitem{HirschSmale1974} 
Hirsch MW, Smale S (1974) Differential Equations, Dynamical Systems and Linear Algebra. Academic Press: New York.

\bibitem{Dietz1976} 
Dietz K (1976) The incidence of infectious diseases under the influence of seasonal fluctuations. In: Mathematical models in medicine. Springer Berlin Heidelberg, pp 1--15

\bibitem{Li2017} 
Li J, Teng Z, Wang G, Zhang L, Hu C (2017) Stability and bifurcation analysis of an SIR epidemic model with logistic growth and saturated treatment. Chaos Solitons Fractals 99: 63--71

\bibitem{ZhangChen1999} 
Zhang XA, Chen L (1999) The periodic solution of a class of epidemic models. Comput. Math. Appl. 38: 61--71

\bibitem{Li2011} 
Li J, Blakeley D, Smith RJ (2011) The failure of $\mathcal{R}_0$. Comput. Math. Methods Med. 2011: 17 pages

\bibitem{TangChen2001} 
Tang S, Chen L (2001) A discrete predator-prey system with age-structure for predator and natural barriers for prey. Math. Model. Numer. Anal. 35: 675--690


\bibitem{Choisy2006} 
Choisy M, Gu\'egan J-F, Rohani P (2006) Dynamics of infectious diseases and pulse vaccination: Teasing apart the embedded resonance effects. Phys. D 223: 26--35



\bibitem{WangYoung2003} 
Wang Q, Young LS (2003) Strange Attractors in Periodically-Kicked Limit Cycles and Hopf Bifurcations. Commun. Math. Phys. 240: 509--529


\bibitem{IoossJoseph1980} 
Iooss G, Joseph D (1980) Elementary Stability and Bifurcation Theory. New York. Springer

\bibitem{Yongzhen2011} 
Yongzhen P, Shuping L, Changguo L, Chen S (2011) The effect of constant and pulse vaccination on an SIR epidemic model with infectious period. Appl. Math. Model. 35: 3866--3878

\bibitem{ZouGaoZhong2009} 
Zou Q, Gao S, Zhong Q (2009) Pulse Vaccination Strategy in an Epidemic Model with Time Delays and Nonlinear Incidence. Adv. Stud. Biol. 1: 307--321

\bibitem{Herrera2023} 
Herrera AR (2023) Paradoxical phenomena and chaotic dynamics in epidemic models subject to vaccination. Commun. Pure Appl. Anal. 19: 2533--2548

\bibitem{Bonotto2008} 
Bonotto EM, Federson M (2008) Limit sets and the Poincar\'e-Bendixson Theorem in impulsive semidynamical systems. J. Differ. Equ. 244: 2334--2349

\bibitem{WangYoung2002} 
Wang Q, Young LS (2002) From Invariant Curves to Strange Attractors. Commun. Math. Phys. 225:275--304

\bibitem{Herman1977} 
Herman MR (1997) Mesure de Lebesgue et Nombre de Rotation. Lecture Notes in Mathematics. 597: 271--293, Springer, Berlin


\bibitem{Shilnikov2Turaev2004} 
Shilnikov A, Shilnikov L, Turaev D (2004) On some mathematical topics in classical synchronization: A
tutorial. Int. J. Bifurc. Chaos Appl. 14: 2143--2160

\bibitem{Zanolin2010} 
Margheri A, Rebelo C, Zanolin F (2010) Chaos in periodically perturbed planar Hamiltonian systems using linked twist maps. J. Differ. Equ. 249: 3233--3257 

\bibitem{Anishchenko1993}  
Anishchenko VS, Safonova MA, Chua LO (1993) Confirmation of the Afraimovich-Shilnikov torus-breakdown theorem via a torus circuit. IEEE Trans. Circuits Syst. I Fundam. Theory Appl. 40: 792--800

\bibitem{Anishchenko2007}  
Anishchenko VS, Astakhov V, Neiman A, Vadivasova T, Schimansky-Geier L (2007) Nonlinear dynamics of chaotic and stochastic systems. Springer Series in Synergetics (second). Springer, Berlin

\bibitem{Rodrigues_2022_JDDE}  
Rodrigues AA  (2022) Unfolding a Bykov attractor: from an attracting torus to strange attractors. Journal of Dynamics and Differential Equations, 34(2), 1643--1677.


\bibitem{Bajiya2022} 
Bajiya VP, Tripathi JP, Kakkar V, Kang, Y (2022) Modeling the impacts of awareness and limited medical resources on the epidemic size of a multi-group SIR epidemic model. International Journal of Biomathematics, 15(07), 2250045.

\bibitem{Goel2023}  
 Goel S, Bhatia SK, Tripathi JP, Bugalia S, Rana M, Bajiya VP (2023) SIRC epidemic model with cross-immunity and multiple time delays. J. Math. Biol. 87: 52 pages

\bibitem{Kumar2021}  
Kumar U, Mandal PS, Tripathi JP, Bajiya, VP,  Bugalia, S (2021) SIRS epidemiological model with ratio-dependent incidence: Influence of preventive vaccination and treatment control strategies on disease dynamics. Mathematical Methods in the Applied Sciences, 44(18), 14703--14732.

\bibitem{Kumar2022}  
 Kumar S, Xu C, Ghildayal N, Chandra C, Yang M (2022) Social media effectiveness as a humanitarian response to mitigate influenza epidemic and COVID-19 pandemic. Ann. Oper. Res. 319: 823--851









%
%
%
%


%
%
%
%

%
%

%

%

%
%
%
%

%
%


















\end{thebibliography}
\end{document}